\documentclass[11pt]{amsart}
\title[Generalized deformations   of  solvmanifolds]
{Generalized deformations  and holomorphic Poisson cohomology  of  solvmanifolds}

\author{Hisashi Kasuya}

\usepackage{amssymb}
\usepackage{amsmath}
\usepackage{amscd}
\usepackage{amstext}
\usepackage{amsfonts}
\usepackage[all]{xy}

\theoremstyle{plain}

\theoremstyle{plain}

\theoremstyle{plain}

\theoremstyle{plain}
\newtheorem{theorem}{Theorem}[section] 
\theoremstyle{remark}
\newtheorem{remark}{Remark}
\theoremstyle{Main result}
\newtheorem{main result}{Main result}
\theoremstyle{lemma}
\newtheorem{lemma}[theorem]{Lemma}
\theoremstyle{definition}

\theoremstyle{proposition}
\newtheorem{proposition}[theorem]{Proposition}
\theoremstyle{corollary}
\newtheorem{corollary}[theorem]{Corollary}
\theoremstyle{remark}
\newtheorem{example}{Example}
\theoremstyle{assumption}
\newtheorem{Assumption}[theorem]{Assumption} 

\address[Hisashi Kasuya]{Department of Mathematics, Tokyo Institute of Technology, 1-12-1, O-okayama, Meguro, Tokyo 152-8551, Japan}
\email{kasuya@math.titech.ac.jp}


\newcommand{\C}{\mathbb{C}}
\newcommand{\R}{\mathbb{R}}
\newcommand{\N}{\mathbb{N}}
\newcommand{\Q}{\mathbb{Q}}
\newcommand{\Z}{\mathbb{Z}}
\newcommand{\g}{\frak{g}}
\newcommand{\n}{\frak{n}}

\DeclareMathOperator{\Kur}{Kur}

\begin{document} 

\maketitle
\begin{abstract}
We describe the generalized Kuranishi spaces of solvmanifolds with left-invariant complex structures. 
By using such description, we study the stability of left-invariantness  of deformed generalized complex structures and smoothness of generalized Kuranishi spaces
on certain classes  of solvmanifolds. 
We also  give explicit finite dimensional cohcain complexes which computes the holomorphic Poisson cohomology of nilmanifolds and solvmanifolds.
\end{abstract}
\section{Introduction}
Let $G$ be a simply connected nilpotent Lie group with a lattice (i.e., cocompact discrete subgroup) $\Gamma$.
We call $G/\Gamma$ nilmanifold.
There are many studies on  deformations of  left-invariant complex structures on nilmanifolds.
In particular, there is the good result on the stability  of the  left-invariantness under deformations.
By the result in \cite{RO} (see also \cite{CFP}, \cite{MPPS}), 
if a nilmanifold $G/\Gamma$ with  a left-invariant complex structure $J$ satisfies the condition $(\ast)$"The Dolbeault cohomology can be completely computed by the Lie algebra of  $G$",  then the   Kuranishi space (locally complete family)  of the complex structure $J$ consists  entirely of left-invariant complex structures and thus every sufficiently small deformation of $J$ is equivalent to a left-invariant complex structure.
It is known that the class of nilmanifolds with left-invariant complex structures satisfying the condition $(\ast)$ is large (see Theorem  \ref{Nilcod} and \cite{ROc}).
Moreover, it is conjectured that every nilmanifold with left-invariant complex structure satisfies  the condition $(\ast)$.
Hence we can expect that every sufficiently small deformation of  a left-invariant complex structure on every nilmanifold is equivalent to a left-invariant complex structure.

Let $G$ be a simply connected solvable Lie group with a lattice  $\Gamma$.
We call $G/\Gamma$ solvmanifold.
Unlike nilmanifolds,  in \cite{Hd}, on a $3$-dimensional non-nilpotent complex parallelizable solvmanifold (called   Nakamura manifold \cite{Na}), Hasegawa remarks that there exists a small deformation   which is not equivalent to any left-invariant complex structure.
In general, we can not describe  the Kuranishi space of a solvmanifold $G/\Gamma$ with a left-invariant complex structure by only the Lie algebra of $G$.
We are interested in describing the Kuranishi spaces of a solvmanifolds with  left-invariant complex structures explicitly  and giving criteria for deciding whether every sufficiently small deformation is equivalent to a left-invariant complex structure.

The singularity  of the Kuranishi spaces of  nilmanifolds with left-invariant complex structures   is  also an interesting problem.
It is known that the Kuranishi spaces  of compact K\"ahler   manifolds with the trivial  canonical  bundles are smooth.
Nilmanifolds with left-invariant complex structures have the  trivial  canonical  bundles  but do not admit K\"ahler structures  unless they are tori (see \cite{BG}, \cite{H}).
Hence Kuranishi spaces of  nilmanifolds with left-invariant complex structures may be singular.
In \cite{ROP}, Rollenske studies the  singularity  of the Kuranishi spaces of complex parallelizable  nilmanifolds.

In this paper, we study generalized deformations i.e.,  deformations of complex structures as generalized complex structures as in \cite{Gua}.
We can construct the locally complete family (generalized Kuranishi space) for generalized  deformations which is an extension of the ordinary Kuranishi space (see Section \ref{GENK} and \cite{Gua}).
Let $(M,J) $ be a compact complex manifold.
We consider the space $A^{0,\ast}(M, \bigwedge T_{1,0}M)$ of $(0,\ast)$-differential forms with values in the holomorphic tangent poly-vector bundle with the  Schouten bracket $[\bullet]$ and the Dolbeault operator $\bar\partial$.
We denote 
\[dG^{r}(M,J)=\bigoplus_{p+q=r}A^{0,q}(M, \bigwedge^{p} T_{1,0}M).\]
Then $(dG^{\ast}(M,J), \wedge, [\bullet],  \bar\partial)$ is a differential Gerstenhaber algebra (shortly DGA).
The generalized Kuranishi space for generalized deformations of $J$ is controlled by the DGA $dG^{\ast}(M,J)$.
By DGA-techniques,  we can easily extend the Rollenske's results (see Proposition \ref{NILDE}, \ref{Cut}).

The main purpose of this paper is to  study  generalized deformations of left-invariant complex structures on solvmanifolds by using DGA-techniques.
We consider the following two classes of solvmanifolds with  left-invariant complex structures.
\begin{itemize}
\item  solvmanifolds of splitting type. (Section \ref{sss}) 
\item complex parallelizable solvmanifolds. (Section \ref{ParL})
\end{itemize}
They are large classes which contain Nakamura manifolds.
For a solvmanifold $G/\Gamma$ in these classes, we prove that we obtain  an explicit finite dimensional sub-DGA $C_{\Gamma}^{\ast}\subset dG^{\ast}(G/\Gamma,J)$ such that the inclusion induces a cohomology isomorphism (Theorem \ref{DGAI}, \ref{COMPAR}).
By this result, we can describe the generalized Kuranishi space explicitly (Corollary \ref{DEFCST11} ,\ref{DEFCST}). 
By this description, we obtain criterions for the stability  of the  left-invariantness under generalized deformations by using  certain characters on $G$ (Theorem \ref{lefiviv}, \ref{linpp}).
Moreover, we can give estimations of  the singularities of the generalized Kuranishi spaces of certain solvmanifolds (Theorem \ref{smsppp}, Corollary \ref{2dee}).

Let $(M,J)$ be a compact complex manifold.
A bi-vector field $\mu\in C^{\infty}(\bigwedge^{2}T_{1,0}M)$ is called holomorphic Poisson if $\bar\partial \mu=0$ and $[\mu\bullet \mu]=0$.
Let $\mu$ be a holomorphic Poisson bi-vector field $\mu$ on $M$.
Then we can define the cohomology determined by $\mu$ which is isomorphic to the Lie algebroid cohomology of generalized complex structure given by $\mu$ (see \cite{L-G}).
Unlike real Poisson cohomology, this cohomology is  always finite dimensional (see \cite{L-G}).
Our results on DGAs can be applied to holomorphic Poisson cohomology
on complex solvmanifolds in the above two classes, we can compute the holomorphic Poisson cohomology by explicit finite-dimensional complexes (Corollary \ref{posp}, \ref{popar})
By this result, we can find a remarkable  example. (see Section \ref{nakho})

\section{Notation and Conventions}

$\bullet$ "DGA" means "differential Gerstenhaber algebra".

$\bullet$ "DGrA" means "differential graded algebra".
(We must distinguish DGA and DGrA.)

$\bullet$  "DBiA" means "differential bi-graded algebra".

$\bullet$ We write $[n]=\{1,2,\dots, n\}=\{a\in \Z_{>0}\vert a\le n\}$.

$\bullet$ For a multi-index $I=\{i_{1},\dots,i_{p}\}$, products through $I$ are written shortly, for examples, $\alpha_{I}=\alpha_{i_{1}}\cdot\alpha_{i_{2}}\cdots\alpha_{i_{p}}$, $x_{I}=x_{i_{1}}\wedge x_{i_{2}}\wedge\dots \wedge x_{i_{p}}$ etc.

\section{DGAs of Lie algebras}
Let $G$ be a simply connected  Lie group with a left-invariant complex structure $J$.
and $\g$  the Lie algebra of $G$.
Consider the decomposition $\g\otimes \C=\g_{1,0}\oplus \g_{0,1}$ where $\g_{1,0}$ (resp. $\g_{0,1}$) is the $\sqrt{-1}$-eigenspace (resp. $\sqrt{-1}$-eigenspace) of $J$.
Then we can define the bracket $[\bullet]$ on $\g_{1,0}\oplus \g_{0,1}^{\ast}$ 
as 
\[\left[X+\bar x\bullet Y+\bar y\right]=[X,Y]+i_{X}d\bar y-i_{Y}d\bar x
\]
for $X,Y\in  \g_{1,0}$, $\bar x,\bar y\in \g_{0,1}^{\ast}$ where $d$ is the exterior differential which is  dual to the Lie bracket.
Extending this bracket on $\textstyle\bigwedge  (\g_{1,0}\oplus \g_{1,0}^{\ast})$ and considering the Dolbeault operator $\bar\partial$, $(\textstyle\bigwedge  (\g_{1,0}\oplus \g_{1,0\ast}),\;\wedge,\; [\bullet],\;\bar\partial)$ is a DGA.
We denote this DGA by $dG^{\ast}(\g, J)$.
\begin{remark}\label{repa}
Suppose $(G,J)$ is a complex Lie group.
Then since $\g_{1,0}$ and $\g_{0,1}^{\ast}$ consist of holomorphic vector fields and anti-holomorphic forms respectively, we have $i_{X}d\bar y=0$ and $\bar\partial X=0$ for $X\in \g_{1,0}$ and $\bar y\in \g_{0,1}^{\ast}$.
Hence  regarding $\textstyle\bigwedge  \g_{1,0}$ as a DGA with trivial differential and $\textstyle\bigwedge  \g_{0,1}^{\ast}$ as a DGA with trivial bracket, we have 
\[dG^{\ast}(\g,J)=\textstyle\bigwedge  \g_{1,0}\otimes \textstyle\bigwedge  \g_{0,1}^{\ast}
\]
as a tensor product of DGAs.
In particular we have
\[[ \textstyle\bigwedge ^{p}\g_{1,0} \otimes \bigwedge^{q}\g_{0,1}^{\ast}\bullet \bigwedge^{0}\g_{1,0} \otimes \bigwedge^{r}\g_{0,1}^{\ast}]=0.\]
\end{remark}

We suppose $G$ has a lattice (i.e. cocompact discrete subgroup) $\Gamma$.
Then we have the inclusion $dG^{\ast}(\g, J)\subset dG^{\ast}(G/\Gamma,J)$ of DGA.

We use the following theorem:
\begin{theorem}{\rm  (\cite[Theorem 1.3]{Sal}, \cite[Theorem 3.1]{CG})}\label{Cano}
Let $G$ be a simply connected $2n$-dimensional nilpotent Lie group with a left invariant structure $J$.
Then we have $d(\textstyle\bigwedge ^{n} \g_{1,0}^{\ast})=0$.
In particular for a lattice $\Gamma$, the canonical bundle $\textstyle\bigwedge  ^{n}T^{\ast}_{1,0}G/\Gamma$ of the nilmanifold $G/\Gamma$ is trivial.
\end{theorem}

By this theorem we have:

\begin{corollary}\label{nildoll}
Let $G$ be a simply connected $2n$-dimensional nilpotent Lie group with a left invariant structure $J$.
We suppose that the inclusion $\textstyle\bigwedge ^{\ast}\g_{1,0}^{\ast} \otimes \bigwedge^{\ast}\g_{0,1}^{\ast}\subset A^{\ast,\ast}(G/\Gamma)$ induces an isomorphism 
\[H^{\ast,\ast}_{\bar\partial}(\g)\cong H^{\ast,\ast}_{\bar\partial }(G/\Gamma).\]
Then the inclusion $dG^{\ast}(\g, J)\subset dG^{\ast}(G/\Gamma,J)$ induces a cohomology isomorphism.
\end{corollary}
\begin{proof}
Fix $p\in \N$.
We consider the natural $\C$-linear isomorphism $\textstyle\bigwedge ^{p} \g_{1,0} \cong \bigwedge^{n-p} \g^{\ast}_{1,0} $  given by the pairing
\[\textstyle\bigwedge ^{p} \g^{\ast}_{1,0} \times \bigwedge^{n-p} \g^{\ast}_{1,0}\to \bigwedge^{n} \g^{\ast}_{1,0}.\]
By theorem \ref{Cano}, this isomorphism induces an isomorphism
$\textstyle\bigwedge ^{p} T_{1,0}G/\Gamma \cong \bigwedge^{n-p} T^{\ast}_{1,0}G/\Gamma$
of holomorphic vector bundles, and hence induces an isomorphism 
\[(A^{0,\ast}(G/\Gamma,\textstyle\bigwedge ^{p} T_{1,0}G/\Gamma),\bar\partial)\cong (A^{n-p,\ast}(G/\Gamma),\bar\partial).\]
Now we have the commutative diagram
\[\xymatrix{
	\textstyle\bigwedge ^{n-p}\g_{1,0}^{\ast}\otimes \bigwedge^{\ast} \g_{0,1}\ar[r]\ar[d]^{\cong}& A^{n-p,\ast}(G/\Gamma) \ar[d]^{\cong} \\
\bigwedge^{p}\g_{1,0}\otimes \bigwedge ^{\ast}\g^{\ast}_{0,1} \ar[r]&A^{0,\ast}(G/\Gamma, \bigwedge^{p} T_{1,0}G/\Gamma). 
 }
\]
Since the inclusion $\textstyle\bigwedge ^{n-p}\g_{1,0}^{\ast}\otimes \bigwedge^{\ast} \g_{0,1}\subset A^{n-p,\ast}(G/\Gamma) $ induces a cohomology isomorphism, the inclusion  \[\textstyle\bigwedge ^{p}\g_{1,0}\otimes \bigwedge ^{\ast}\g^{\ast}_{0,1}\subset A^{0,\ast}(G/\Gamma, \bigwedge^{p} T_{1,0}G/\Gamma)\] does so.
Hence the corollary follows.
\end{proof}

\begin{theorem}\label{Nilcod}
Let $G$ be a simply connected nilpotent Lie group with a lattice $\Gamma$ and a left-invariant complex structure $J$.
Then
the inclusion $\textstyle\bigwedge ^{\ast}\g_{1,0}^{\ast} \otimes \bigwedge^{\ast}\g_{0,1}^{\ast}\subset A^{\ast,\ast}(G/\Gamma)$ induces an isomorphism 
\[H^{\ast,\ast}_{\bar\partial}(\g)\cong H^{\ast,\ast}_{\bar\partial }(G/\Gamma),\]
 if $(G,J,\Gamma)$ meet one of the following  conditions:

(N) The complex manifold $(G/\Gamma, J)$ has the structure of an iterated principal holomorphic torus bundle (\cite{CF}).

(Q)  $J$ is a rational complex structure i.e. for the rational structure $\g_{\Q}\subset \g$ of the Lie algebra $\g$ induced by a lattice $\Gamma$ (see \cite[Section 2]{R}) we have $J(\g_{\Q})\subset \g_{\Q}$ (\cite{CFI}).

 (C)   $(G,J)$ is a complex  Lie group (\cite{Sak}).\\
\end{theorem}

\begin{remark}
It is known that the isomorphism $H^{\ast,\ast}_{\bar\partial}(\g)\cong H^{\ast,\ast}_{\bar\partial }(G/\Gamma)$ holds on  an open set of any connected component of the moduli space  of left-invariant complex structures on a nilmanifold $G/\Gamma$.
(see \cite{CFI}).
\end{remark}

\section{Generalized deformations}
\subsection{DGA and generalized Kuranishi space}\label{GENK}
Let $(M,J) $ be a compact $n$-dimensional complex manifold.
We define the generalized complex structure ${\mathcal J}\in {\rm End}(TM\oplus T^{\ast}M)$ by ${\mathcal J}=\left(
\begin{array}{cc}
-J& 0  \\
0&    J^{\ast}
\end{array}
\right)$
with the maximally isotropic subspace $L_{\mathcal J}=T_{0,1}M\oplus T_{1,0}^{\ast}M$ in $(TM\oplus T^{\ast}M)\otimes \C$.

For $\epsilon \in C^{\infty}(\textstyle\bigwedge ^{2}  L^{\ast})=C^{\infty}(\bigwedge^{2} \bar L)=dG^{2}(M,J)$, we have the new generalized complex structure given by  the maximally isotropic subspace
\[L_{\epsilon}=(1+\epsilon )L_{\mathcal J}=\{E+i_{E}\epsilon\vert E\in L_{\mathcal J}\}\]
if $\epsilon$ satisfies the generalized Maurer-Cartan equation
\[\bar\partial \epsilon+\frac{1}{2}[\epsilon\bullet\epsilon]=0.\]

For a Hermitian metric on M, 
we consider the $\C$-anti-linear Hodge star operator \[\bar\ast_{g}\colon A^{0,q}(M,\textstyle\bigwedge  ^{p}T_{1,0}M)\\
\to A^{n,n-q}(M,\bigwedge ^{p}T_{1,0}^{\ast}M),\]
the adjoint differential operator $\bar\partial^{\ast}=-\bar\ast_{g}\bar\partial\bar\ast_{g}$,
the Laplacian operator $\Box_{g}=\bar\partial^{\ast}\bar\partial+\bar\partial\bar\partial^{\ast}$,
the space ${\mathcal H}_{g}^{\ast}(M,J)={\rm Ker} \, \Box_{g}{\vert_{dG^{\ast}(M, J)}}$ of harmonic forms,
the orthogonal projection $H\colon dG^{\ast}(M, J)\to {\mathcal H}_{g}^{\ast}(M,J)$ and the Green operator $G\colon dG^{\ast}(M, J)\to dG^{\ast}(M, J)$.
Take a basis $\eta_{1},\dots \eta_{k}$ of ${\mathcal H}_{g}^{\ast}(M,J)$.
For parameters $t=(t_{1},\dots ,t_{k})$,  we consider the formal power series $\phi(t)=\phi(t_{1},\dots, t_{k})$ with values in  $dG^{2}(G/\Gamma, J)$ given inductively by $\phi_{1}(t)=\sum_{i=1}^{k} t_{i}\eta_{i}$ and
\[\phi_{r}(t)=\frac{1}{2}\sum_{s=1}^{r-1}\bar\partial^{\ast}G[\phi_{s}(t)\bullet\phi_{r-s}(t)].\]
For sufficiently small $\delta>0$, for $\vert t\vert<\delta$, $\phi(t)$ converges.
We denote
\[\Kur^{gen}(M,J)=\left\{t=(t_{1},\dots ,t_{k})\vert \; \vert t\vert<\delta,\;  H\left(\left[\phi(t)\bullet\phi(t)\right]\right)=0\right\}.\]
Then we have:
\begin{theorem}{\rm (\cite{Gua})}\label{GDEF}
For $t\in \Kur^{gen}(M,J)$, $\phi(t)$ satisfies the generalized Maurer-Cartan equation.
Any sufficiently  small deformation of the generalized complex structure  $\mathcal J$ is equivalent to a generalized complex structure given by  the maximally isotropic subspace
\[L_{\phi(t)}=(1+\phi(t) )L_{\mathcal J}\]
for some $t \in \Kur^{gen}(M,J)$.
\end{theorem}

Suppose the canonical bundle $\textstyle\bigwedge ^{n}T^{\ast}_{1,0}M$ is trivial.
Since  we have an isomorphism $\textstyle\bigwedge  ^{p}T_{1,0}^{\ast}M\cong \bigwedge ^{n-p}T_{1,0}M$ by the pairing $\wedge\colon \textstyle\bigwedge  ^{p}T_{1,0}^{\ast}M\times \bigwedge ^{n-p}T^{\ast}_{1,0}M\to \bigwedge ^{n}T^{\ast}_{1,0}M$,  we can identifies $(A^{n,\ast}(M,\textstyle\bigwedge ^{p}T_{1,0}^{\ast}M),\bar\partial)$ with $(A^{0,\ast}(M,\bigwedge^{n-p}T_{1,0}M),\bar\partial)$.
Hence we can regard 
\[\bar\ast_{g}\colon A^{0,q}(M,\textstyle\bigwedge ^{p}T_{1,0}M)\to A^{0,n-q}(M,\bigwedge^{n-p}T_{1,0}M )\] and so 
\[\bar\ast_{g}\colon dG^{r}(M,J)\to  dG^{2n-r}(M,J).
\]
\begin{lemma}\label{RED}
Suppose the canonical bundle $\textstyle\bigwedge  ^{n}T^{\ast}_{1,0}M$ is trivial.
Let $C^{\ast}$ be a finite-dimensional subDGA of $dG(M,J)$ such that $\bar\ast (C^{\ast}) \subset C^{\ast}$.
Suppose the inclusion $C^{\ast}\subset dG^{\ast}(M,J)$ induces a cohomology isomorphism.
Then for $t\in \Kur^{gen}(M,J)$, we have 
\[\phi(t)\in  C^{2}.\]

\end{lemma}
\begin{proof}
By $\bar\ast_{g} (C^{\ast}) \subset C^{\ast}$, we can regard $\partial^{\ast}_{\vert_{C^{\ast}}}$ and $\Box_{g\vert_{C^{\ast}}}$ as operators on $C^{\ast}$.
Since  $C^{\ast}$ is finite dimensional, we can easily show the decomposition
\[C^{\ast}={\rm Ker}\, \Box_{g\vert_{C^{\ast}}}\oplus{\rm Im}\, \partial_{\vert_{C^{\ast}}}\oplus {\rm Im}\, \partial^{\ast}_{\vert_{C^{\ast}}}.
\]
Hence the induced map $H^{\ast}(C^{\ast})\to H^{\ast}(dG^{\ast}(M,J))$ is represented by the inclusion ${\rm Ker}\, \Box_{\vert_{C^{\ast}}}\subset {\mathcal H}^{\ast}_{g}(M,J)$.
Since  the induced map $H^{\ast}(C^{\ast})\to H^{\ast}(dG^{\ast}(M,J))$ is an isomorphism, we have ${\rm Ker}\, \Box_{\vert_{C^{\ast}}}={\mathcal H}^{\ast}_{g}(M,J)$.
For $t_{i} \eta_{i}\in {\mathcal H}^{2}_{g}(M,J)={\rm Ker}\, \Box_{\vert_{C^{\ast}}}$,
since $C^{\ast}$ is  a finite-dimensional sub-DGA of $dG^{\ast}(M,J)$, for a basis $\zeta_{1},\dots, \zeta_{l}$ of $C^{2}$, 
we have $\phi(t)=\sum \psi_{i}(t) \zeta_{i}$ which converges for $\vert t\vert <\delta$.
Hence the lemma follows.
\end{proof}
By this lemma we can prove the generalized version of Rollenske's result in \cite{RO}.

\begin{proposition}\label{NILDE}
Let $G$ be a simply connected $2n$-dimensional nilpotent Lie group with a left invariant complex structure $J$.
We suppose that the inclusion $\textstyle\bigwedge ^{\ast}\g_{1,0}^{\ast}\otimes \bigwedge^{\ast} \g^{\ast}_{0,1}\subset A^{\ast,\ast}(G/\Gamma)$ induces an isomorphism 
\[H^{\ast,\ast}_{\bar\partial}(\g)\cong H^{\ast,\ast}_{\bar\partial }(G/\Gamma).\]
Let $\mathcal J$ be the generalized complex structure given by $J$.
Then any sufficiently small deformation of $\mathcal J$ is equivalent to a generalized complex structure which is induced by a left-invariant generalized complex structure on $G$.
\end{proposition}
\begin{proof}
Take a basis $X_{1},\dots,X_{n}$ of $\g_{1,0}$.
We consider the left-invariant Hermitian metric $g=\sum x_{i}\bar x_{i} $ where $x_{1},\dots,x_{n}$ is the basis of $\g_{1,0}^{\ast}$.
Then we have $\bar\ast_{g}(dG^{\ast}(\g,J))\subset dG^{\ast}(\g,J)$.
For the basis 
\[X_{1},\dots,X_{n},\bar X_{1},\dots,\bar X_{n}, x_{1},\dots,x_{n},\bar x_{1},\dots,\bar x_{n}\]
 of $(\g\oplus\g^{\ast})\otimes\C$ we write
\[{\mathcal J}= \sqrt{-1}\sum ( X_{i}\otimes x_{i}+\bar x_{i}\otimes \bar X_{i})- \sqrt{-1}\sum ( \bar X_{i}\otimes \bar x_{i}+x_{i}\otimes  X_{i}).\]
By Theorem \ref{GDEF} and Lemma \ref{RED}, any sufficiently small deformation of $\mathcal J$ is equivalent to
\[{\mathcal J}_{\epsilon}= \sqrt{-1}\sum ( X_{i}^{\epsilon}\otimes x^{\epsilon}_{i}+(\bar x_{i})^{\epsilon}\otimes (\bar X_{i})^{\epsilon})- \sqrt{-1}\sum ( \bar{X}^{\bar\epsilon}_{i}\otimes \bar{x}^{\bar\epsilon}_{i}+x^{\bar\epsilon}_{i}\otimes  X^{\bar\epsilon}_{i})\]
where we write $E^{\epsilon}=E+i_{E}\epsilon$ for $E\in L_{\mathcal J}$.
Hence the theorem follows.
\end{proof}

\subsection{Smoothness}
We can prove   generalized version of Rollenske's result in \cite{ROP}.
\begin{proposition}\label{Cut}
Let $G$ be a $n$-dimensional simply connected complex nilpotent Lie group with a lattice $\Gamma$.
Suppose $G$ is $\nu$-step i.e. $C_{\nu-1}\g\not=0$ and $C_{\nu}\g=0$ where we denote by $\g=C_{0}\g\supset C_{1}\g\supset C_{2} \g\cdots$  the lower central series.
Then the formal power series $\phi(t)$ as in Section \ref{GENK} is a finite sum
\[\phi(t)=\sum_{i=1}^{\nu} \phi_{i}(t)
\]
and we have
\[H\left[\phi(t),\phi(t)\right]=\sum_{i+j\le \nu} H\left[\phi_{i}(t),\phi_{j}(t)\right].
\]
In particular  $\Kur^{gen}(G/\Gamma, J)$ is cut out by polynomial equations of degree 
at most $\nu$.
\end{proposition}
\begin{proof}
By remark \ref{repa}, we have
\begin{multline*}
\left[dG^{2}(\g,J)\bullet dG^{2}(\g,J)\right]\\
\subset \left[\textstyle\bigwedge ^{1}\g_{1,0} \otimes \bigwedge^{1}\g_{0,1}^{\ast}\bullet \bigwedge^{1}\g_{1,0} \otimes \bigwedge^{1}\g_{0,1}^{\ast}\right] 
\bigoplus  \left[\bigwedge^{2}\g_{1,0} \otimes \bigwedge^{0}\g_{0,1}^{\ast}\bullet \bigwedge^{1}\g_{1,0} \otimes \bigwedge^{1}\g_{0,1}^{\ast}\right]\\
\bigoplus \left[\bigwedge^{2}\g_{1,0} \otimes \bigwedge^{0}\g_{0,1}^{\ast}\bullet \bigwedge^{2}\g_{1,0} \otimes \bigwedge^{0}\g_{0,1}^{\ast}\right].
\end{multline*}
We have $\bar\partial^{\ast}(\textstyle\bigwedge ^{p}\g_{1,0})=0$ for any $p\in \N$.
Moreover since we have $\bar\partial_{\vert_{\textstyle\bigwedge ^{n-1} \g^{\ast}_{0,1}}}=0$ by the unimodularity of $G$, we have  $\bar\partial^{\ast}(\textstyle\bigwedge ^{p}\g_{1,0} \otimes \bigwedge^{1}\g_{0,1}^{\ast})=0$ for  any $p\in \N$.
Hence we have 
\[\bar\partial^{\ast} G \left[dG^{2}(\g,J)\bullet dG^{2}(\g,J)\right]\subset \left[\textstyle\bigwedge ^{1}\g_{1,0} \otimes \bigwedge^{1}\g_{0,1}^{\ast}\bullet \bigwedge^{1}\g_{1,0} \otimes \bigwedge^{1}\g_{0,1}^{\ast}\right].\]
By this, for  $i\ge 2$ inductively we have
\[\phi_{i}(t)\in C_{i-1}\g_{1,0}\otimes \textstyle\bigwedge ^{1}\g_{0,1}^{\ast}
\]
for  $i\ge 2$.
Hence we have $\phi_{\nu+1}(t)=0$ and $\left[\phi_{i}(t),\phi_{j}(t)\right]\in C_{i+j-1}\g_{1,0}\otimes \textstyle\bigwedge ^{1}\g_{0,1}^{\ast}$ for $i,j\ge 2$ and $\left[\phi_{\nu}(t),\phi_{1}(t)\right]=0$.
These imply the theorem.
\end{proof}

We consider the special condition  which implies that $\Kur^{gen}(M,J)$ is smooth.
In Section \ref{FORMA}, we give a large class of complex solvmanifolds satisfying such condition.
\begin{proposition}\label{Forsm}
Let $(M,J)$ be a compact complex manifold.
Suppose there exists a Hermitian metric $g$ on $M$ such that the space ${\mathcal H}_{g}^{\ast}(M,J)$ is closed under the Schouten bracket.
Take  a basis $\eta_{1},\dots \eta_{k}$ of ${\mathcal H}_{g}^{\ast}(M,J)$.
Then we have
\[\Kur^{gen}(M,J)=\left\{ t=(t_{1},\dots, t_{k})\vert \left[\sum_{i=1}^{k} t_{i}\eta_{i}\bullet \sum_{i=1}^{k} t_{i}\eta_{i}\right]=0\right\}.\]
In particular   $\Kur^{gen}(G/\Gamma, J)$ is cut out by polynomial equations of degree 
at most $2$.

Moreover, suppose that the Schouten bracket on ${\mathcal H}_{g}^{\ast}(M,J)$ is trivial.
Then  $\Kur^{gen}(G/\Gamma, J)$ is smooth.
\end{proposition}
\begin{proof}
For any parameter $ t=(t_{1},\dots, t_{k})$ we have $\left[\sum_{i=1}^{k} t_{i}\eta_{i}\bullet\sum_{i=1}^{k} t_{i}\eta_{i}\right]\in {\mathcal H}_{g}^{\ast}(M,J)$.
Hence we have $G[\sum_{i=1}^{k} t_{i}\eta_{i}\bullet\sum_{i=1}^{k} t_{i}\eta_{i}]=0$ and so we have
$\phi(t)=\sum_{i=1}^{k} t_{i}\eta_{i}$.
This implies the first assertion of the proposition.
 Obviously, the second assertion follows from the first assertion.

\end{proof}

\begin{remark}
For a compact complex manifold, we consider the differential graded Lie algebra 
\[A^{0,\ast}(M,  T_{1,0}M)\]
of differential forms with values in the holomorphic tangent bundle and the space
\[{\mathcal H}_{g}^{\ast}(M,J)\cap A^{0,\ast}(M,  T_{1,0}M)\]
of harmonic forms which belong to such space.
We take a basis   a basis $\eta_{1},\dots,\eta_{j},\dots \eta_{k}$ of ${\mathcal H}_{g}^{2}(M,J)$ such that $\eta_{1},\dots,\eta_{j}$ is a basis of ${\mathcal H}_{g}^{2}(M,J)\cap A^{0,1}(M,  T_{1,0}M)$.
Then the subspace
\[\Kur(M,J)=\left\{t=(t_{1},\dots ,t_{j},0,\dots, 0)\vert \; \vert t\vert<\delta,\;  H([\phi(t)\bullet\phi(t)])=0\right\}\]
of $\Kur(M,J)^{gen}$ is the usual Kuranishi space (see \cite{Ku} and \cite{Gua}
).
Hence study of generalized deformation covers study of usual deformation of complex structures.
\end{remark}

\section{Cohomology of holomorphic Poisson manifolds}
Let $(M,J)$ be a compact complex manifold.
A bi-vector field $\mu\in C^{\infty}(\textstyle\bigwedge ^{2}T_{1,0}M)$ is called holomorphic Poisson if $\bar\partial \mu=0$ and $\left[\mu\bullet \mu\right]=0$.
Let $\mu$ be a holomorphic Poisson bi-vector field $\mu$ on $M$.
Then $T^{\ast}M$ is naturally a holomorphic Lie algebroid.
Since $\mu\in dG^{\ast}(M,J)$ satisfies the generalized Maurer-Cartan equation, we have the deformed generalized structure given by  the maximally isotropic subspace
$L_{\mu}$.
Considering the DBiA $\left(A^{0,\ast}(M,\textstyle\bigwedge ^{\ast}T_{1,0}M),\bar\partial\right)$, by the differential operator $[\mu\bullet ]\colon A^{0,\ast}(M,\textstyle\bigwedge ^{\ast}T_{1,0}M)\to A^{0,\ast}(M,\bigwedge^{\ast+1}T_{1,0}M)$, we have the double complex $(A^{0,\ast}(M,\textstyle\bigwedge ^{\ast}T_{1,0}M),\bar\partial,[\mu\bullet ])$.
Then we consider the following three cohomologies:

$\bullet$ The Lie algebroid cohomology of the holomorphic Lie algebroid $T^{\ast}M$.

$\bullet$ The Lie algebroid cohomology of the algebroid $L_{\mu}$ (Lie algeboid cohomology of a generalized complex manifold see \cite{Gua}).

$\bullet$ The total cohomology of the double complex 
\[(A^{0,\ast}(M,\textstyle\bigwedge ^{\ast}T_{1,0}M),\bar\partial,[\mu\bullet ]).\]
It is known that, these cohomologies are all isomorphic (see \cite{L-G}).
We use the notation $H^{\ast}(M,\mu)$ for any one of the above cohomology groups and call it the holomorphic Poisson cohomology of $(M, J, \mu)$.

\begin{lemma}\label{holpo}
Let $C^{\ast,\ast}$ be a sub-DBiA of the DBiA $(A^{0,\ast}(M,\textstyle\bigwedge ^{\ast}T_{1,0}M),\bar\partial)$.
We suppose that the inclusion  $C^{\ast,\ast}\subset A^{0,\ast}(M,\textstyle\bigwedge ^{\ast}T_{1,0}M)$ induces a $\bar\partial$-cohomology isomorphism and  for $C^{r}=\bigoplus_{p+q=r}C^{p,q}$, $C^{\ast}$ is a sub-DGA of $dG^{\ast}(M,J)$.
Let $\mu\in C^{2,0}$ be a holomorphic Poisson bi-vector field.
Then the inclusion $C^{\ast,\ast}\subset A^{0,\ast}(M,\textstyle\bigwedge ^{\ast}T_{1,0}M)$ induces an isomorphism between the total cohomology of $(C^{\ast,\ast},\bar\partial,[\mu\bullet ])$ and the total cohomology of  \[(A^{0,\ast}(M,\textstyle\bigwedge ^{\ast}T_{1,0}M),\bar\partial,[\mu\bullet ]).\]
Hence the total cohomology of $(C^{\ast,\ast},\bar\partial,[\mu\bullet ])$ is isomorphic to $H^{\ast}(M,\mu)$.
\end{lemma}
\begin{proof}
For the double complexes $(C^{\ast,\ast},\bar\partial,[\mu\bullet ])$ and
 \[(A^{0,\ast}(M,\textstyle\bigwedge ^{\ast}T_{1,0}M),\bar\partial,[\mu\bullet ]),\]
 we have the spectral sequences $E^{\ast,\ast}_{\ast}(C^{\ast,\ast})$ and 
$E_{\ast}^{\ast,\ast}(A^{0,\ast}(M,\textstyle\bigwedge ^{\ast}T_{1,0}M))$
 such that $E^{\ast,\ast}_{1}(C^{\ast,\ast})\cong H_{\bar\partial}^{\ast,\ast}(C^{\ast,\ast})$ and  
\[E_{1}^{\ast,\ast}(A^{0,\ast}(M,\textstyle\bigwedge ^{\ast}T_{1,0}M))\cong H_{\bar\partial}^{\ast,\ast}(A^{0,\ast}(M,\textstyle\bigwedge ^{\ast}T_{1,0}M)) .\]
Since the inclusion  $C^{\ast,\ast}\subset A^{0,\ast}(M,\textstyle\bigwedge ^{\ast}T_{1,0}M)$ induces a $\bar\partial$-cohomology isomorphism, the inclusion $C^{\ast,\ast}\subset A^{0,\ast}(M,\textstyle\bigwedge ^{\ast}T_{1,0}M)$ induces an isomorphism $E^{\ast,\ast}_{1}(C^{\ast,\ast})\cong E_{1}^{\ast,\ast}(A^{0,\ast}(M,\textstyle\bigwedge ^{\ast}T_{1,0}M))$.
Hence by \cite[Theorem 3.5]{Mc}, the lemma follows.
\end{proof}
By this lemma we have:
\begin{corollary}\label{PONIL}
Let $G$ be a simply connected $2n$-dimensional nilpotent Lie group with a left invariant structure $J$.
We suppose that the inclusion $\textstyle\bigwedge ^{\ast}\g_{1,0}^{\ast} \otimes \bigwedge^{\ast}\g_{0,1}^{\ast}\subset A^{\ast,\ast}(G/\Gamma)$ induces an isomorphism 
\[H^{\ast,\ast}_{\bar\partial}(\g)\cong H^{\ast,\ast}_{\bar\partial }(G/\Gamma).\]
Let $\mu\in \textstyle\bigwedge ^{2}\g_{1,0} $ be a holomorphic Poisson bi-vector field.
Then the inclusion $\textstyle\bigwedge ^{\ast}\g_{1,0} \otimes \bigwedge^{\ast}\g_{0,1}^{\ast}\subset A^{0,\ast}(G/\Gamma,\bigwedge^{\ast}T_{1,0}G/\Gamma)$ induces an isomorphism between the total cohomology of $(\textstyle\bigwedge ^{\ast}\g_{1,0} \otimes \bigwedge^{\ast}\g_{0,1}^{\ast},\bar\partial,[\mu\bullet ])$ and the total cohomology of  $(A^{0,\ast}(G/\Gamma,\bigwedge^{\ast}T_{1,0}G/\Gamma),\bar\partial,[\mu\bullet ])$.
Hence the total cohomology of $(\textstyle\bigwedge ^{\ast}\g_{1,0} \otimes \bigwedge^{\ast}\g_{0,1}^{\ast},\bar\partial,[\mu\bullet ])$ is isomorphic to $H^{\ast}(G/\Gamma,\mu)$. 
\end{corollary}
\begin{proof}
In the proof of Theorem \ref{nildoll}, we showed that the inclusion 
\[\textstyle\bigwedge ^{\ast}\g_{1,0} \otimes \bigwedge^{\ast}\g_{0,1}^{\ast}\subset A^{0,\ast}(G/\Gamma,\bigwedge^{\ast}T_{1,0}G/\Gamma)\]
 induces an $\bar\partial$-cohomology isomorphism.
Thus we can apply Lemma \ref{holpo}.
\end{proof}

\begin{example}
Consider the $3$-dimensional complex  Heisenberg group
 \[N=\left\{ \left(
\begin{array}{ccc}
1&  a&   c \\
0&     1& b\\
0& 0& 1 
\end{array}
\right)\colon a, b, c\in \C \right\}\]
and $G= N\times \C$.
Then $G$ admits a lattice $\Gamma$.
We take a basis $X,Y,Z,W$ of $\g_{1,0}$ such that $[X,Y]=Z$ and other brackets are $0$, and a basis $\bar x, \bar y, \bar z, \bar w$ of $\g^{\ast}_{0,1}$ such that $d\bar z=-\bar x\wedge \bar y$ and $d\bar x=d\bar y=d\bar w=0$.
We consider the following holomorphic Poisson bi-vector field
\[\mu=X\wedge Z+ Y\wedge W.\]
Then by Corollary \ref{PONIL}, $H^{\ast}(G/\Gamma,\mu)$ is isomorphic to the total cohomology of $(\textstyle\bigwedge ^{\ast}\g_{1,0} \otimes \bigwedge^{\ast}\g_{0,1}^{\ast},\bar\partial,[\mu\bullet ])$.

$\mu$ is induced by a holomorphic symplectic form on $G/\Gamma$.
By the non-degeneracy of $\mu$, we can show that
the total cohomology of $(\textstyle\bigwedge ^{\ast}\g_{1,0} \otimes \bigwedge^{\ast}\g_{0,1}^{\ast},\bar\partial,[\mu\bullet ])$ is isomorphic to the total cohomology of  $(\textstyle\bigwedge ^{\ast}\g_{1,0}^{\ast} \otimes \bigwedge^{\ast}\g_{0,1}^{\ast},\bar\partial,\partial)$
and hence by Nomizu's theorem in \cite{Nom}, it is isomorphic to the de Rham cohomology $H^{\ast}(G/\Gamma,\C)$ of $G/\Gamma$.
The cohomology  $H^{\ast}(G/\Gamma,\mu)$ is not isomorphic to the Dolbeault cohomology $H^{\ast}(dG^{\ast}(G/\Gamma,J))$.
In Section \ref{FORMA}, we give examples of solvmanifolds with holomorphic symplectic forms such that the holomoprhic Poisson cohomologies are isomorphic to the Dolbeault cohomologies.
In Section \ref{ParL}, we give an example of a solvmanifold with a holomorphic Poisson bi-vector field such that the holomoprhic Poisson cohomology is neither the de Rham cohomology nor the Dolbeault cohomology.

\end{example}

\section{DGAs and generalized deformations of  solvmanifolds of splitting type}\label{sss}
\subsection{DGA}
In this section, we consider:
\begin{Assumption}\label{Ass}
We consider pairs $(G,\Gamma)$ where 
$G$ is the semi-direct product $\C^{n}\ltimes _{\phi}N$ with a left-invariant complex structure $J=J_{\C}\oplus J_{N}$ and the following conditions are satis􏰃ed.
\begin{enumerate}
\item $N$ is a simply connected nilpotent Lie group with a left-invariant complex structure $J_{N}$.\\
Let $\frak a$ and $\n$ be the Lie algebras of $\C^{n}$ and $N$ respectively.\\
\item For any $t\in \C^{n}$, $\phi(t)$ is a holomorphic automorphism of $(N,J_{N})$.\\
\item $\phi$ induces a semi-simple action on the Lie algebra $\n$ of $N$.\\
\item $G$ has a lattice $\Gamma$. (Then $\Gamma$ can be written by $\Gamma=\Gamma^{\prime}\ltimes_{\phi}\Gamma^{\prime\prime}$ such that $\Gamma^{\prime}$ and $\Gamma^{\prime\prime}$ are  lattices of $\C^{n}$ and $N$ respectively and for any $t\in \Gamma^{\prime}$ the action $\phi(t)$ preserves $\Gamma^{\prime\prime}$.) \\
\item The inclusion $\textstyle\bigwedge ^{\ast}\n_{1,0}^{\ast}\otimes \bigwedge^{\ast} \n_{0,1}^{\ast}\subset A^{\ast,\ast}(N/\Gamma^{\prime\prime})$ induces an isomorphism 
\[H^{\ast,\ast}_{\bar\partial}(\n)\cong H^{\ast,\ast}_{\bar\partial }(N/\Gamma^{\prime\prime}).\]
\end{enumerate}
\end{Assumption}

Consider the decomposition $\n\otimes {\C}=\n_{1,0}\oplus \n_{0,1}$.
By the condition (2), this decomposition is a direct sum of $\C^{n}$-modules.
By the condition (3) we have a basis $Y_{1},\dots ,Y_{m}$ of $\n^{1,0}$ such that the action $\phi$ on $\n_{1,0}$ is represented by
$\phi(t)={\rm diag} (\alpha_{1}(t),\dots, \alpha_{m} (t))$.
Since $Y_{j}$ is a left-invariant vector field on $N$,
the vector field $\alpha_{j}Y_{j}$ on $\C^{n}\ltimes _{\phi} N$ is  left-invariant.
Hence we have a basis $X_{1},\dots,X_{n}, \alpha_{1}Y_{1},\dots ,\alpha_{m}Y_{m}$ of $\g_{1,0}$.
Let $x_{1},\dots,x_{n}, \alpha^{-1}_{1}y_{1},\dots ,\alpha_{m}^{-1}y_{m}$ be the  basis of $\g^{\ast}_{1,0}$ which is dual to $X_{1},\dots,X_{n}, \alpha_{1}Y_{1},\dots ,\alpha_{m}Y_{m}$.
Then we have 
\begin{multline*}
\textstyle\bigwedge  ^{p}\g_{1,0}^{\ast}\otimes\bigwedge^{q}\g^{\ast}_{0,1}\\
=\bigwedge ^{p}\langle x_{1},\dots ,x_{n}, \alpha^{-1}_{1}y_{1},\dots ,\alpha^{-1}_{m}y_{m}\rangle\otimes \bigwedge ^{q}\langle \bar x_{1},\dots ,\bar x_{n}, \bar\alpha^{-1}_{1}\bar y_{1},\dots ,\bar\alpha^{-1}_{m}\bar y_{m}\rangle.
\end{multline*}

\begin{lemma}{\rm (\cite[Lemma 2.2]{Kd})}\label{charr}
Let  $\alpha\colon\C^{n}\to \C^{\ast}$ be a  $C^{\infty}$-character of $\C^{n}$.
There exists a unique unitary character $\beta$ such that $\alpha \beta^{-1}$ is holomorphic.
\end{lemma}

By this lemma take the unique unitary characters $\beta_{i}$ and $\gamma_{i}$ on $\C^{n}$ such that $\alpha_{i}\beta_{i}^{-1}$ and $\bar\alpha\gamma^{-1}_{i}$ are holomorphic.

\begin{theorem}{\rm (\cite[Corollary 4.2]{Kd})}\label{CORR}
Let  $B^{\ast,\ast}_{\Gamma}\subset A^{\ast,\ast}(G/\Gamma)$ be the sub-DBiA of $A^{\ast,\ast}(G/\Gamma)$ given by
\[B^{p,q}_{\Gamma}=\left\langle x_{I}\wedge \alpha^{-1}_{J}\beta_{J}y_{J}\wedge \bar x_{K}\wedge \bar \alpha^{-1}_{L}\gamma_{L}\bar y_{L}{\Big \vert} \begin{array}{cc}\vert I\vert+\vert J\vert=p,\, \vert K\vert+\vert L\vert=q \\ (\beta_{J}\gamma_{L})_{\vert_{\Gamma}}=1\end{array}\right\rangle.
\]
Then  the inclusion $B^{\ast,\ast}_{\Gamma}\subset A^{\ast,\ast}(G/\Gamma)$ induces a cohomology isomorphism
\[H^{\ast,\ast}_{\bar \partial}(B^{\ast,\ast}_{\Gamma})\cong H^{\ast,\ast}_{\bar \partial}(G/\Gamma).
\]

\end{theorem}


In this paper we consider the following assumption.
\begin{Assumption}\label{unim}
$\alpha_{[m]}=\alpha_{1}\cdot\alpha_{2}\cdots \alpha_{m}=1$ where we write $[m]=\{1,2,\dots, m\}$.
\end{Assumption}
In this assumption, by Theorem \ref{Cano}, we have $dy_{[m]}=0$
and hence the holomorphic  canonical bundle of $G/\Gamma$ is trivialized by the global holomorphic frame $x_{[n]}\wedge y_{[m]}$.
Then we have:
\begin{theorem}\label{DGAI}
Let $(G,\Gamma)$  be  as in Assumption \ref{Ass} with Assumption \ref{unim}.
We define the subspace 
\[C^{p,q}_{\Gamma}=\left\langle X_{I}\wedge \alpha_{J}\beta^{-1}_{J}Y_{J}\otimes \bar x_{K}\wedge \bar \alpha^{-1}_{L}\gamma_{L}\bar y_{L}{\Big \vert} \begin{array}{cc}\vert I\vert+\vert J\vert=p,\, \vert K\vert+\vert L\vert=q \\ (\beta^{-1}_{J}\gamma_{L})_{\vert_{\Gamma}}=1\end{array}\right\rangle
\]
of $A^{0,q}(G/\Gamma, \textstyle\bigwedge ^{p} T^{1,0}G/\Gamma)$.
We denote $C^{k}_{\Gamma}=\oplus_{p+q=k}C^{p,q}_{\Gamma}$.
Then  $(C^{\ast}_{\Gamma},\bar\partial )$ is a sub-DGA of $dG^{\ast}(G/\Gamma,J)$
and the  inclusion $C^{\ast}_{\Gamma}\subset dG^{\ast}(G/\Gamma,J)$ induces a cohomology isomorphism.
\end{theorem}
\begin{proof}

We  consider the weight decomposition
\[\textstyle\bigwedge  (\n_{1,0}\oplus \n_{0,1}^{\ast})=\bigoplus V_{\epsilon_{i}}
\]
of the $\C^{n}$-action via $\phi$ .
Since  $\phi(t)$ induces a semi-simple automorphism on the DGA $\textstyle\bigwedge  \left(\n_{1,0}\oplus \n_{0,1}^{\ast}\right)$ for any $t\in \C^{n}$, we have $V_{\epsilon_{i}}\wedge V_{\epsilon_{j}}\subset V_{ \epsilon_{i}\epsilon_{j}}$, $\left[V_{\epsilon_{i}}\bullet V_{\epsilon_{j}}\right]\subset V_{ \epsilon_{i}\epsilon_{j}}$ and $\bar\partial(V_{\epsilon_{i}})\subset V_{\epsilon_{i}}$.
Taking the unitary character $\zeta_{i}$ of $\C^{n}$ such that $\epsilon_{i}\zeta_{i}^{-1}$ is holomorphic as Lemma \ref{charr},  we have
\[C^{\ast}_{\Gamma}= \bigoplus_{(\zeta_{i})_{\vert_{\Gamma}}=1}\textstyle\bigwedge  ((\C^{n})_{1,0}\oplus (\C^{n})_{0,1}^{\ast})\otimes \epsilon_{i}\zeta_{i}^{-1} V_{\epsilon_{i}}.
\]
Hence $C^{\ast}_{\Gamma}$ is closed under wedge product.
Since $\epsilon_{i}\zeta_{i}^{-1} $ is holomorphic,
we have
\begin{multline*}
\left[\textstyle\bigwedge  ((\C^{n})_{1,0}\oplus (\C^{n})_{0,1}^{\ast})\otimes \epsilon_{i}\zeta_{i}^{-1} V_{\epsilon_{i}}\bullet \bigwedge ((\C^{n})_{1,0}\oplus (\C^{n})_{0,1}^{\ast})\otimes \epsilon_{j}\zeta_{j}^{-1} V_{\epsilon_{j}}\right]\\
\subset \bigwedge ((\C^{n})_{1,0}\oplus (\C^{n})_{0,1}^{\ast})\otimes \epsilon_{i}\epsilon_{j}\zeta_{i}^{-1}\zeta_{j}^{-1} V_{\epsilon_{i}\epsilon_{j}}
\end{multline*}
and 
\begin{multline*}\bar\partial \left(\textstyle\bigwedge  ((\C^{n})_{1,0}\oplus (\C^{n})_{0,1}^{\ast})\otimes \epsilon_{i}\zeta_{i}^{-1} V_{\epsilon_{i}}\right)\\
=\bigwedge ((\C^{n})_{1,0}\oplus (\C^{n})_{0,1}^{\ast})\otimes \epsilon_{i}\zeta_{i}^{-1} \bar\partial V_{\epsilon_{i}}\subset \bigwedge ((\C^{n})_{1,0}\oplus (\C^{n})_{0,1}^{\ast})\otimes \epsilon_{i}\zeta_{i}^{-1} V_{\epsilon_{i}}.
\end{multline*}
Hence $C^{\ast}_{\Gamma}$ is a subDGA of $dG^{\ast}(G/\Gamma,J)$.

We will show that $C^{\ast}_{\Gamma}\subset dG^{\ast}(G/\Gamma,J)$ induces a cohomology isomorphism.
Since the holomorphic  canonical bundle of $G/\Gamma$ is trivialized by the global holomorphic frame $x_{[n]}\wedge y_{[m]}$,
 we have the isomorphism
\[\textstyle\bigwedge  ^{p} T_{1,0}G/\Gamma \cong \bigwedge ^{n+m-p} T_{1,0}^{\ast}G/\Gamma
\]
which is given by 
\[\textstyle\bigwedge  ^{p}\g_{1,0}\ni X_{I}\wedge\alpha_{J}Y_{J}\mapsto 
x_{I}\wedge \alpha^{-1}_{[m]-J}y_{J}\in \bigwedge ^{n+m-p}\g_{1,0}^{\ast}.\]
By this isomorphism we have the commutative diagram
\[\xymatrix{
	B^{n+m-p,\ast}_{\Gamma}\ar[r]\ar[d]^{\cong}& A^{n+m-p,\ast}(G/\Gamma) \ar[d]^{\cong} \\
C^{p,\ast}_{\Gamma} \ar[r]&A^{0,\ast}(G/\Gamma, \bigwedge^{p} T_{1,0}G/\Gamma). 
 }
\]
Hence by Theorem \ref{CORR}, $C^{p,\ast}_{\Gamma}$ is a subcomplex of $A^{0,\ast}(G/\Gamma, \textstyle\bigwedge ^{p} T_{1,0}G/\Gamma)$ and the inclusion $C^{p,\ast}_{\Gamma} \subset A^{0,\ast}(G/\Gamma,\textstyle\bigwedge ^{p} T_{1,0}G/\Gamma)$ induces a cohomology isomorphism for each $p$ and this implies  the theorem.
\end{proof}

\subsection{Holomorphic Poisson cohomology}
We consider the cohomology of holomorphic Poisson solvmanifolds.
\begin{corollary}\label{posp}
Let $(G,\Gamma)$  be  as in Assumption \ref{Ass} with Assumption \ref{unim}.
We consider the DBiA $(C^{\ast,\ast}_{\Gamma},\bar\partial)$ as Theorem \ref{DGAI}.
Let $\mu \in C^{2.0}_{\Gamma}$ be a holomorphic Poisson bi-vector field.
Then the inclusion 
\[C^{\ast,\ast}_{\Gamma}\subset A^{0,\ast}(G/\Gamma, \textstyle\bigwedge ^{\ast} T_{1,0}G/\Gamma)\] induces an isomorphism between the total cohomology of the double complex $(C^{\ast,\ast}_{\Gamma},\bar\partial, [\mu\bullet])$ and the total cohomology of the double complex $(A^{0,\ast}(G/\Gamma, \textstyle\bigwedge ^{\ast} T_{1,0}G/\Gamma),\bar\partial,  [\mu\bullet])$.
Hence total cohomology of the double complex $(C^{\ast,\ast}_{\Gamma},\bar\partial, [\mu\bullet])$ is isomorphic to $H(G/\Gamma, \mu)$.
\end{corollary}
\begin{proof}
In the proof of Theorem \ref{DGAI}, we showed that the inclusion 
\[C^{\ast,\ast}_{\Gamma}\subset A^{0,\ast}(G/\Gamma, \textstyle\bigwedge ^{\ast} T_{1,0}G/\Gamma)\]
 induces a $\bar\partial$-cohomology isomorphism.
Thus we can apply Lemma \ref{holpo}.
\end{proof}
\subsection{Generalized deformation}\label{spdeff}
Let $(G,\Gamma)$  be  as in Assumption \ref{Ass} with Assumption \ref{unim}.
We consider the left-invariant Hermitian metric 
\[g=\sum x_{i} \bar x_{i}+ \sum \alpha_{i}^{-1}\bar\alpha^{-1}_{i}y_{i}\bar y_{i}.
\]
Then the $\C$-anti-linear Hodge star operator \[\bar\ast_{g}\colon A^{0,q}(G/\Gamma,\textstyle\bigwedge  ^{p}T_{1,0}G/\Gamma)\to A^{0,n+m-q}(G/\Gamma,\bigwedge ^{n+m-p}T_{1,0}G/\Gamma)
\]
is given by 
\[
\bar \ast_{g} (X_{I}\wedge \alpha_{J}Y_{J}\otimes \bar x_{K}\wedge \bar \alpha^{-1}_{L}\bar y_{L})
=X_{[n]-I}\wedge\alpha_{[m]-J}Y_{[m]-J}\otimes \bar x_{[n]-K}\wedge \bar \alpha^{-1}_{[m]-L}\bar y_{[m]-L}.
\]
Since we have $\alpha_{[m]-J}=\alpha^{-1}_{J}$ and $\beta_{[m]-J}=\beta^{-1}_{J}$, we have
\begin{multline*}\bar \ast_{g} ( X_{I}\wedge \alpha_{J}\beta^{-1}_{J}Y_{J}\otimes \bar x_{K}\wedge \bar \alpha^{-1}_{L}\gamma_{L}\bar y_{L})\\
=X_{[n]-I}\wedge\alpha_{[m]-J}\beta_{J}Y_{[m]-J}\otimes \bar x_{[n]-K}\wedge \bar \alpha^{-1}_{[m]-L}\gamma_{L}^{-1}\bar y_{[m]-L}\\
=X_{[n]-I}\wedge\alpha_{[m]-J}\beta_{[m]-J}^{-1}Y_{[m]-J}\otimes \bar x_{[n]-K}\wedge \bar \alpha^{-1}_{[m]-L}\gamma_{[m]-L}\bar y_{[m]-L} \in C^{\ast}_{\Gamma}.
\end{multline*}
Hence we have $\bar\ast_{g}(C^{\ast})\subset C^{\ast}$.
For the above $X_{1},\dots X_{n},Y_{1},\dots ,Y_{m}$, $x_{1},\dots x_{n}, y_{1},\dots,y_{m}$, as an endomorphism of  the complexified  tangent bundle,
 the complex structure $J$  can be written as
\[J=\sqrt{-1}\sum(X_{i}\otimes x_{i}+Y_{i}\otimes y_{i})-\sqrt{-1}\sum(\bar X_{i}\otimes \bar x_{i}+\bar Y_{i}\otimes \bar y_{i}).\]
and this gives the generalized complex structure 
\begin{multline*}
{\mathcal J}=\sqrt{-1}\sum (X_{i}\otimes x_{i}+Y_{i}\otimes y_{i}+\bar x_{i}\otimes \bar X_{i}+\bar y_{i}\otimes \bar Y_{i})\\
-\sqrt{-1}\sum(\bar X_{i}\otimes\bar x_{i}+\bar Y_{i}\otimes\bar y_{i}+ x_{i}\otimes  X_{i}+ y_{i}\otimes  Y_{i}).
\end{multline*}
By Theorem \ref{DGAI}, Theorem \ref{GDEF} and Lemma \ref{RED}, we have:
\begin{corollary}\label{DEFCST11}
Any sufficiently small deformation of the generalized complex structure $\mathcal J$ which is given by $J$ is equivalent to
\begin{multline*}{\mathcal J}_{\epsilon}
=\sqrt{-1}\sum (X^{\epsilon}_{i}\otimes x^{\epsilon}_{i}+Y^{\epsilon}_{i}\otimes y^{\epsilon}_{i}+(\bar x_{i})^{\epsilon}\otimes (\bar X_{i})^{\epsilon}+(\bar y_{i})^{\epsilon}\otimes (\bar Y_{i})^{\epsilon})\\
-\sqrt{-1}\sum((\bar X_{i})^{\bar\epsilon}\otimes(\bar x_{i})^{\bar\epsilon}+(\bar Y_{i})^{\bar\epsilon}\otimes(\bar y_{i})^{\bar\epsilon}+ x_{i}^{\bar\epsilon}\otimes  X_{i}^{\bar\epsilon}+ y_{i}^{\bar\epsilon}\otimes  Y_{i}^{\bar\epsilon}).
\end{multline*}
for some $\epsilon \in C^{2}_{\Gamma}$ where we write $E^{\epsilon}=E+i_{E}\epsilon$ for $E\in L_{\mathcal J}$.
\end{corollary}

We consider the following condition:
\\
\ 

(${\mathcal D}_{r}$) For $r\in \N$, for any $J,L\subset [m]$ with $\vert J\vert+\vert L\vert\le r$,  $(\beta^{-1}_{J}\gamma_{L})_{\vert_{\Gamma}}=1$ if and only if $\alpha_{J}\bar\alpha^{-1}_{L}=1$.
\\
\ \\
If the condition ${\mathcal D}_{r}$ holds, then for $p,q\in N$ with $p+q\le r$ we have
\[C^{p,q}_{\Gamma}=\left\langle X_{I}\wedge Y_{J}\otimes \bar x_{K}\wedge \bar y_{L} \vert \alpha_{J}\alpha^{-1}_{L}=1  \right\rangle.\]
Hence we have $C^{\ast}_{\Gamma}\subset dG^{\ast}(\g,J)$.
This implies that $\phi(t)\in \Kur^{gen}(G/\Gamma)$  gives a left-invariant generalized complex structure on $G/\Gamma$.
We have:
\begin{theorem}\label{lefiviv}
Let $(G,\Gamma)$  be  as in Assumption \ref{Ass} with Assumption \ref{unim}.
Suppose that the condition (${\mathcal D}_{2}$) holds.
Then  any sufficiently small deformation of the generalized complex structure $\mathcal J$ which is given by $J$ is  equivalent to a  generalized complex structure on $G/\Gamma$ which is induced by a left-invariant generalized complex structure on $G$.
\end{theorem}

\begin{theorem}\label{smsppp}
Let $(G,\Gamma)$  be  as in Assumption \ref{Ass} with Assumption \ref{unim}.
We assume the following conditions:

(a) $(N,J)$ is a  $\nu$-step complex nilpotent Lie group.

(b) For any $J,L\subset [m]$,  $(\beta^{-1}_{J}\gamma_{L})_{\vert_{\Gamma}}=1$ if and only if $(\alpha_{J}\bar\alpha^{-1}_{L})_{\vert_{\Gamma}}=1$.
 \\ 
Then the formal power series $\phi(t)$ as in Section \ref{GENK} is a finite sum
\[\phi(t)=\sum_{i=1}^{\nu} \phi_{i}(t)
\]
and we have
\[H[\phi(t),\phi(t)]=\sum_{i+j\le \nu} H[\phi_{i}(t),\phi_{j}(t)].
\]
In particular  $\Kur^{gen}(G/\Gamma, J)$ is cut out by polynomial equations of degree 
at most $\nu$.
\end{theorem}
\begin{remark}
The condition (b) is more general than the condition (${\mathcal D}_{r}$).
\end{remark}
\begin{proof}
If $(\alpha_{J}\bar\alpha^{-1}_{L})_{\vert_{\Gamma}}=1$, then $\alpha_{J}\bar\alpha^{-1}_{L}$ is unitary and we have $\alpha_{J}\bar\alpha^{-1}_{L}=\beta_{J}\gamma^{-1}_{L}$.
By the assumption (b), we have
\[C^{p,q}_{\Gamma}=\left\langle X_{I}\wedge Y_{J}\otimes \bar x_{K}\wedge \bar y_{L} \vert (\alpha_{J}\bar\alpha^{-1}_{L})_{\vert_{\Gamma}}=1  \right\rangle.
\]
Since $Y_{1},\dots ,Y_{m}$ and $\bar y_{1},\dots,\bar y_{m}$ are 	bases of $\n_{1,0}$ and $\n_{0,1}^{\ast}$ respectively,
 we have an embedding $\psi\colon C^{s}_{\Gamma}\hookrightarrow dG^{s}(\C^{n}\oplus \n, J)$. 
Consider the left-invariant Hermitian metric 
\[ h=\sum x_{i} \bar x_{i}+ \sum y_{i}\bar y_{i}\]
on the simply connected complex nilpotent Lie group $\C^{n}\times N$.
Then for the left-invariant Hermitian metric 
\[g=\sum x_{i} \bar x_{i}+ \sum \alpha_{i}^{-1}\bar\alpha^{-1}_{i}y_{i}\bar y_{i}
\]
on $G$, we have $\psi\circ \bar\ast_{g}=\bar\ast_{h}\circ \psi$.
Hence the formal power series $\phi(t)$ as in Section \ref{GENK} is constructed by the DGA $dG^{s}(\C^{n}\oplus \n, J)$.
As the proof of  Theorem \ref{Cut}, we can show the theorem. 
\end{proof}
\begin{example}
Let $G=\C\ltimes_{\phi} N$ such that $N=N_{1}\times N_{2}$, both $N_{1}$ and $N_{2}$ are the $3$-dimensional complex  Heisenberg groups and the action $\phi$ is given by 
\begin{multline*}
\phi(x+\sqrt{-1}y)\left( \left(
\begin{array}{ccc}
1&  w_{1,1}&   w_{1,3} \\
0&     1& w_{1,2}\\
0& 0& 1 
\end{array}
\right),
\left(
\begin{array}{ccc}
1&  w_{2,1}&   w_{2,3} \\
0&     1& w_{2,2}\\
0& 0& 1 
\end{array}
\right)\right)\\
=\left( \left(
\begin{array}{ccc}
1&  e^{k_{1}x}w_{1,1}&e^{(k_{1}+k_{2})x}   w_{1,3} \\
0&     1&e^{k_{2}x} w_{1,2}\\
0& 0& 1 
\end{array}
\right),
\left(
\begin{array}{ccc}
1&  e^{-k_{1}x}w_{1,1}&e^{-(k_{1}+k_{2})x}   w_{1,3} \\
0&     1&e^{-k_{2}x} w_{1,2}\\
0& 0& 1 
\end{array}
\right)\right)
\end{multline*}
where $k_{1}$ and $k_{2}$ are positive integers.
Taking  a basis $Y_{1,1},Y_{1,2},Y_{1,3},Y_{2,1},Y_{2,2},Y_{2,3}$ of 
$\n_{1,0}$ such that 
\[[Y_{1,1}, Y_{1,2}]=Y_{1,3},\,\,\,\,\, [Y_{2,1}, Y_{2,2}]=Y_{2,3},
\]
the action $\phi$ on $\n_{1,0}$ is represented by
\begin{multline*}
diag(\alpha_{1,1},\alpha_{1,2},\alpha_{1,3},\alpha_{2,1},\alpha_{2,2},\alpha_{2,3})\\
=diag(e^{k_{1}x},e^{k_{2}x},e^{(k_{1}+k_{2})x},e^{-k_{1}x},e^{-k_{2}x},e^{-(k_{1}+k_{2})x}).
\end{multline*}
For the coordinate $z=x+\sqrt{-1}y$, we have
\[\g_{1,0}=\left\langle \frac{\partial}{\partial z}, e^{k_{1}x}Y_{1,1},e^{k_{2}x}Y_{1,2},e^{(k_{1}+k_{2})x}Y_{1,3},e^{-k_{1}x}Y_{2,1},e^{-k_{2}x}Y_{2,2},e^{-(k_{1}+k_{2})x}Y_{2,3}
\right\rangle
\]
and 
\begin{multline*}
\textstyle\bigwedge  \g_{1,0}^{\ast}\otimes \bigwedge \g_{0,1}^{\ast}\\
=\left\langle dz, e^{-k_{1}x} y_{1,1},e^{k_{2}x} y_{1,2},e^{-(k_{1}+k_{2})x} y_{1,3},e^{k_{1}x} y_{2,1},e^{k_{2}x} y_{2,2},e^{(k_{1}+k_{2})x} y_{2,3}\right\rangle\\
\otimes \left\langle d\bar z, e^{-k_{1}x}\bar y_{1,1},e^{k_{2}x}\bar y_{1,2},e^{-(k_{1}+k_{2})x}\bar y_{1,3},e^{k_{1}x}\bar y_{2,1},e^{k_{2}x}\bar y_{2,2},e^{(k_{1}+k_{2})x}\bar y_{2,3}
\right\rangle
\end{multline*}
where $y_{1,1},y_{1,2},y_{1,3},y_{2,1},y_{2,2},y_{2,3}$ is the  dual basis of  $Y_{1,1},Y_{1,2},Y_{1,3},Y_{2,1},Y_{2,2},Y_{2,3}$.
Taking unitary characters
\begin{multline*}
(\beta_{1,1},\beta_{1,2},\beta_{1,3},\beta_{2,1},\beta_{2,2},\beta_{2,3})\\
=(e^{-\sqrt{-1}k_{1}y},e^{-\sqrt{-1}k_{2}y},e^{-\sqrt{-1}(k_{1}+k_{2})y},e^{\sqrt{-1}k_{1}y},e^{\sqrt{-1}k_{2}y},e^{\sqrt{-1}(k_{1}+k_{2})y}),
\end{multline*}
each $\alpha_{i,j}\beta_{i,j}^{-1}$ is holomorphic.

Take a unimodular matrix $B \in  SL(2,Z)$ with distinct positive eigenvalues $\lambda$ and $\lambda$, and set $a =
\log \lambda$. 
Then as similar to \cite{sawai-yamada}, for any non-zero number $b\in \R$, we have a lattice $\Gamma=(a\Z+\sqrt{-1}b\Z)\ltimes \Gamma^{\prime\prime}$ such that
$\Gamma^{\prime\prime}$ is a lattice  of $N$ which is
given by
\begin{multline*}\Gamma^{\prime\prime}
=\left\{ \left( \left(
\begin{array}{ccc}
1&  u_{1,1}+\lambda u_{2,1}&   u_{1,3}+\lambda u_{2,3} \\
0&     1& u_{1,2}+\lambda u_{2,2}\\
0& 0& 1 
\end{array}
\right),\right.\right.\\\left.
\left(   
\begin{array}{ccc}
1&  u_{1,1}+\lambda^{-1} u_{2,1}&   u_{1,3}+\lambda^{-1} u_{2,3} \\
0&     1& u_{1,2}+\lambda^{-1} u_{2,2}\\
0& 0& 1 
\end{array}
\right)\right)\\
\left.{\Big \vert}
u_{i,j}\in \Z+\sqrt{-1}\Z \right\}.
\end{multline*}
Suppose $b\not=\frac{r}{s}\pi$ for any $r,s\in \Z$. 
Then the assumptions in Theorem \ref{lefiviv} and Theroem \ref{smsppp} hold.
Hence any sufficiently small deformation of the generalized complex structure $\mathcal J$ associated with the complex structure of $G/\Gamma$ is  equivalent to a  generalized complex structure on $G/\Gamma$ which is induced by a left-invariant generalized complex structure on $G$
and $\Kur^{gen}(G/\Gamma, J)$ is cut out by polynomial equations of degree 
at most $2$.

We assume $k_{1}=1$, $k_{2}=2$.
Then we have
\[C_{\Gamma}^{1}=\left\langle \frac{\partial}{\partial z}, d\bar z \right\rangle
\]
\begin{multline*}
C_{\Gamma}^{2}
=\left\langle
 Y_{1,1}\wedge Y_{2,1}, Y_{1,2}\wedge Y_{2,2}, Y_{1,3}\wedge Y_{2,3},\right.\\
\frac{\partial}{\partial z}\wedge d\bar z, Y_{1,1}\wedge \bar y_{2,1}, Y_{2,1}\wedge \bar y_{1,1}, Y_{1,2}\wedge \bar y_{2,2}, Y_{2,2}\wedge \bar y_{1,2}, Y_{1,3}\wedge \bar y_{2,3},Y_{2,3}\wedge \bar y_{1,3}\\
 \bar y_{1,1}\wedge \bar y_{2,1}, \bar y_{1,2}\wedge \bar y_{2,2}, \bar y_{1,3}\wedge \bar y_{2,3}\rangle.
\end{multline*}

We consider the Hermitian metric 
\begin{multline*}g=dzd\bar z+e^{-2x}y_{11}\bar y_{11}+e^{-4x}y_{1,2}\bar y_{1,2}+e^{-6x}y_{1,3}\bar y_{1,3}\\
+e^{2x}y_{2,1}\bar y_{2,1}+e^{4x}y_{2,2}\bar y_{2,2}+e^{6x}y_{2,3}\bar y_{2,3}.
\end{multline*}
Then we have
\begin{multline*}
{\mathcal H}^{2}_{g}(G/\Gamma,J)
=\left\langle
 Y_{1,1}\wedge Y_{2,1}, Y_{1,2}\wedge Y_{2,2}, Y_{1,3}\wedge Y_{2,3},\right.\\
\frac{\partial}{\partial z}\wedge d\bar z, Y_{1,1}\wedge \bar y_{2,1}, Y_{2,1}\wedge \bar y_{1,1}, Y_{1,2}\wedge \bar y_{2,2}, Y_{2,2}\wedge \bar y_{1,2},\\
 \bar y_{1,1}\wedge \bar y_{2,1}, \bar y_{1,2}\wedge \bar y_{2,2}\rangle.
\end{multline*}
For parameters $t=(t_{1121},t_{1222},t_{1323},t_{0}, t_{11}^{\overline{21}},t_{21}^{\overline{11}}, t_{12}^{\overline{22}},t_{22}^{\overline{12}},t^{\overline{11}\overline{21}},t^{\bar1\bar2\bar2\bar2})$,
taking 
\begin{multline*}
\phi_{1}(t)=t_{1121}Y_{1,1}\wedge Y_{2,1}+t_{1222}Y_{1,2}\wedge Y_{2,2}+t_{1323}Y_{1,3}\wedge Y_{2,3}\\
+t_{0}\frac{\partial}{\partial z}\wedge d\bar z+ t_{11}^{\overline{21}}Y_{1,1}\wedge \bar y_{2,1}+t_{21}^{\overline{11}}Y_{2,1}\wedge \bar y_{1,1}+ t_{12}^{\overline{22}}Y_{1,2}\wedge \bar y_{2,2}+t_{22}^{\overline{12}}Y_{2,2}\wedge \bar y_{1,2}\\
+t^{\overline{11}\overline{21}}y_{1,1}\wedge \bar y_{2,1}+t^{\overline{12}\overline{22}}\bar y_{1,2}\wedge \bar y_{2,2},
\end{multline*}
we have 
\begin{multline*}
\phi(t)=t_{1121}Y_{1,1}\wedge Y_{2,1}+t_{1222}Y_{1,2}\wedge Y_{2,2}+t_{1323}Y_{1,3}\wedge Y_{2,3}\\
+t_{0}\frac{\partial}{\partial z}\wedge d\bar z+ t_{11}^{\overline{21}}Y_{1,1}\wedge \bar y_{2,1}+t_{21}^{\overline{11}}Y_{2,1}\wedge \bar y_{1,1}+ t_{12}^{\overline{22}}Y_{1,2}\wedge \bar y_{2,2}+t_{22}^{\overline{12}}Y_{2,2}\wedge \bar y_{1,2}\\
+t^{\overline{11}\overline{21}}y_{1,1}\wedge \bar y_{2,1}+t^{\overline{12}\overline{22}}\bar y_{1,2}\wedge \bar y_{2,2}\\
-\frac{1}{2}t_{11}^{\overline{21}}t_{12}^{\overline{22}}Y_{13}Y_{13}\wedge \bar y_{23}-\frac{1}{2}t_{21}^{\overline{11}}t_{22}^{\overline{12}}Y_{23}\wedge \bar y_{13}
\end{multline*}
where $\phi(t)$ is the formal power series $\phi(t)$ as in Section \ref{GENK}.
We have
\begin{multline*}
\Kur^{gen}(M,J)\\
=\{t=(t_{1121},t_{1222},t_{1323},t_{0}, t_{11}^{\overline{21}},t_{21}^{\overline{11}}, t_{12}^{\overline{22}},t_{22}^{\overline{12}},t^{\overline{11}\overline{21}},t^{\bar1\bar2\bar2\bar2})
\\ \vert \; \vert t\vert<\delta,\;  H([\phi(t)\bullet\phi(t)])=0\}.
\end{multline*}
$\Kur^{gen}(M,J)$ is defined by the following quadratic polynomial equations
\[t_{1112}t_{1222}=0,\,\, \, t_{1121}t_{12}^{\overline{22}}=0,\,\,\, t_{1121}t_{22}^{\overline{12}}=0,\,\,\, t_{1222}t_{11}^{\overline{21}}=0,\,\,\, t_{1222}t_{21}^{\overline{11}}=0.
\]
\begin{remark}
In this case, the ordinary Kuranishi space $\Kur(G/\Gamma, J)$ is given by
\[
\Kur(G/\Gamma, J)
=\{t=(0,0,0,t_{0}, t_{11}^{\overline{21}},t_{21}^{\overline{11}}, t_{12}^{\overline{22}},t_{22}^{\overline{12}},0,0) \vert \; \vert t\vert<\delta,\;  H([\phi(t)\bullet\phi(t)])=0\}
\]
and hence $\Kur(G/\Gamma, J)$ is unobstructed.
But $\Kur^{gen}(G/\Gamma, J)$ is obstructed.
\end{remark}

\end{example}

\section{DGAs of complex parallelizable solvmanifolds}\label{ParL}
\subsection{DGA}
Let $G$ be a simply connected $n$-dimensional complex solvable Lie group.
Consider the Lie algebra $\g_{1,0}$ (resp. $\g_{0,1}$) of the left-invariant holomorphic (resp. anti-holomorphic) vector fields on $G$.
Let $N$ be the nilradical  of $G$.
We can take a  simply connected complex nilpotent subgroup $C\subset G$  such that $G=C\cdot N$ (see \cite{dek}).
Since $C$ is nilpotent, the map
\[C\ni c \mapsto ({\rm Ad}_{c})_{s}\in {\rm Aut}(\g_{1,0})\]
is a homomorphism where $({\rm Ad}_{c})_{s}$ is the semi-simple part of ${\rm Ad}_{s}$.

We have a basis $X_{1},\dots,X_{n}$ of $\g_{1,0}$ such that \[({\rm Ad}_{c})_{s}={\rm diag} (\alpha_{1}(c),\dots,\alpha_{n}(c))\] for $c\in C$.
Let $x_{1},\dots, x_{n}$ be the basis of $\g^{\ast}_{1,0}$ which is dual to $X_{1},\dots ,X_{n}$.

\begin{theorem}{\rm (\cite[Corollary 6.2 and its proof]{KDD})}\label{MMTT}
Let $B^{\ast}_{\Gamma}$ be the subcomplex of $(A^{0,\ast}(G/\Gamma),\bar\partial) $ defined as
\[B^{\ast}_{\Gamma}=\left\langle \frac{\bar\alpha_{I}}{\alpha_{I} }\bar x_{I}{\Big \vert}\left(\frac{\bar\alpha_{I}}{\alpha_{I}}\right)_{ \vert_{\Gamma}}=1\right\rangle.
\]
Then the inclusion $B^{\ast}_{\Gamma}\subset A^{0,\ast}(G/\Gamma) $ induces an isomorphism
\[H^{\ast}(B^{\ast}_{\Gamma})\cong H^{0,\ast}(G/\Gamma).
\]
\end{theorem}

By this theorem we prove:

\begin{theorem}\label{COMPAR}
We consider the subspace 
\[C^{p,q}_{\Gamma}= \textstyle\bigwedge ^{p} \g_{1,0} \otimes B^{q}_{\Gamma}\]
 of $A^{0,q}(G/\Gamma, \bigwedge^{p}T_{1,0}G/\Gamma)$.
Denote $C^{r}_{\Gamma}=\bigoplus _{p+q=r}C^{p,q}_{\Gamma}$.
Then 
$C^{\ast}_{\Gamma}$ is a sub-DGA of $dG^{\ast}(G/\Gamma,J)$ and the inclusion induces a cohomology isomorphism.
\end{theorem}
\begin{proof}
By Theorem \ref{MMTT}, the inclusion $C^{\ast,\ast}_{\Gamma}\subset A^{0,\ast}(G/\Gamma, \textstyle\bigwedge ^{\ast}T_{1,0}G/\Gamma)$ induces a cohomology isomorphism
\[ \textstyle\bigwedge ^{\ast} \g_{1,0} \otimes H^{\ast}_{\bar\partial}(B^{\ast}_{\Gamma})\cong \bigwedge^{\ast} \g_{1,0} \otimes H^{0,\ast}_{\bar\partial}(G/\Gamma).\]
Hence it is sufficient to show that $C^{\ast}_{\Gamma}$ is closed under the Schouten bracket.

For any $X\in \g_{1,0}$,
we have
\[X(\alpha_{I}^{-1})=\alpha^{-1}_{I}\alpha_{I} d\alpha^{-1}_{I}(X).\]
Since $\alpha_{I} d\alpha^{-1}_{I}$ is left-invariant, $\alpha_{I} d\alpha^{-1}_{I}(X)$ is constant.
Hence for some constant $c$, we have
\[X(\alpha_{I}^{-1})=c\alpha_{I}^{-1}.\]
Since $\bar\alpha_{I}$ and $\bar x_{I}\in \bigwedge \g_{0,1}^{\ast}$ are anti-holomorphic,
we have
\[L_{X}\left(\frac{\bar\alpha_{I}}{\alpha_{I} }\bar x_{I}\right) =c\frac{\bar\alpha_{I}}{\alpha_{I} }\bar x_{I}.\]
Hence $\textstyle\bigwedge \g_{1,0}^{\ast}\otimes B^{\ast}_{\Gamma} $ is closed under the Schouten bracket.
\end{proof}

\subsection{Holomorphic Poisson cohomology}
Since the inclusion  
\[C^{\ast,\ast}_{\Gamma}\subset A^{0,\ast}(G/\Gamma, \textstyle\bigwedge ^{\ast}T_{1,0}G/\Gamma)\]
 induces a $\bar\partial$-cohomology isomorphism, by Lemma \ref{holpo}, we have:
\begin{corollary}\label{popar}
Let $\mu \in C^{2.0}_{\Gamma}$ be a holomorphic Poisson bi-vector field.
Then the inclusion $C^{\ast,\ast}_{\Gamma}\subset A^{0,\ast}(G/\Gamma, \textstyle\bigwedge ^{\ast} T_{1,0}G/\Gamma)$ induces an isomorphism between the total cohomology of the double complex $(C^{\ast,\ast}_{\Gamma},\bar\partial, [\mu\bullet])$ and the total cohomology of the double complex \[(A^{0,\ast}(G/\Gamma, \textstyle\bigwedge ^{\ast} T_{1,0}G/\Gamma),\bar\partial,  [\mu\bullet]).\]
Hence total cohomology of the double complex $(C^{\ast,\ast}_{\Gamma},\bar\partial, [\mu\bullet])$ is isomorphic to $H(G/\Gamma, \mu)$.
\end{corollary}
\subsection{Generalized deformation}
We consider the Hermitian metric
\[g=\sum x_{i}\bar x_{i}.
\]
Then for $x_{I}\frac{\bar\alpha_{J}}{\alpha_{J} }\bar x_{J}\in \bigwedge \g_{1,0}^{\ast}\otimes B^{\ast}_{\Gamma}$,
since $G$ is unimodular, 
we have
\[\bar\ast_{g}(x_{I}\wedge\frac{\bar\alpha_{J}}{\alpha_{J} }\bar x_{J})=x_{I^{\prime}}\wedge\frac{\alpha_{J}}{\bar\alpha_{J} }\bar x_{J^{\prime}}=x_{I^{\prime}}\wedge\frac{\bar\alpha_{J^{\prime}}}{\alpha_{J^{\prime}} }\bar x_{J^{\prime}}\in \bigwedge\g_{1,0}^{\ast}\otimes B^{\ast}_{\Gamma}
\]
where $I^{\prime}$ and $J^{\prime}$ are complements of $I$ and $J$ respectively.
Hence we have $\bar\ast_{g}(\textstyle\bigwedge  \g_{1,0}^{\ast}\otimes B^{\ast}_{\Gamma})\subset \bigwedge \g_{1,0}^{\ast}\otimes B^{\ast}_{\Gamma}$.
For the basis $X_{1},\dots X_{n}$, $x_{1},\dots x_{n}$, the complex structure on $G$ is 
\[J=\sqrt{-1}\sum(X_{1}\otimes x_{i})-\sqrt{-1}\sum(\bar X_{i}\otimes \bar x_{i}).\]
and this gives the generalized complex structure 
\[{\mathcal J}=\sqrt{-1}\sum (X_{i}\otimes x_{i}+\bar x_{i}\otimes \bar X_{i})-\sqrt{-1}\sum(\bar X_{i}\otimes\bar x_{i}+ x_{i}\otimes  X_{i}).
\] 
By Theorem \ref{COMPAR}, as the last section, we have:
\begin{corollary}\label{DEFCST}
Any sufficiently small deformation of the generalized complex structure $\mathcal J$ which is given by the complex Lie group $G$  is equivalent to
\[{\mathcal J}_{\epsilon}
=\sqrt{-1}\sum (X^{\epsilon}_{i}\otimes x^{\epsilon}_{i}+(\bar x_{i})^{\epsilon}\otimes (\bar X_{i})^{\epsilon})\\
-\sqrt{-1}\sum((\bar X_{i})^{\bar\epsilon}\otimes(\bar x_{i})^{\bar\epsilon}+ x_{i}^{\bar\epsilon}\otimes  X_{i}^{\bar\epsilon}).
\]
for some $\epsilon \in C^{2}_{\Gamma}$ where we write $E^{\epsilon}=E+i_{E}\epsilon$ for $E\in L_{\mathcal J}$.
\end{corollary}

We consider the following condition:
\\
\ 

(${\mathcal E}_{r}$) For $r\in \N$, for any $I\subset [n]$ with $\vert I\vert \le r$,  $\left(\frac{\bar\alpha_{I}}{\alpha_{I}}\right)_{ \vert_{\Gamma}}=1$ if and only if $\alpha_{I}=1$.
\\
\ \\
If the condition (${\mathcal E}_{r}$) holds, then for $p\le r$, we have
\[B^{q}_{\Gamma}=\left\langle \bar x_{I} \vert\alpha_{I}=1,\; \vert I \vert=q \right\rangle
\]
and hence we have $B^{p}_{\Gamma}\subset \textstyle\bigwedge^{p} \g_{0,1}^{\ast}$.
Thus we have:
\begin{theorem}\label{linpp}
Suppose that the condition (${\mathcal E}_{2}$) holds.
Then  any sufficiently small deformation of the generalized complex structure $\mathcal J$ which is given by the complex Lie group $G$  is  equivalent to a  generalized complex structure on $G/\Gamma$ which is induced by a left-invariant generalized complex structure on $G$.
\end{theorem}

\subsection{Generalized deformations and holomorphic Poisson cohomology of Nakamura manifolds}\label{nakho}
\begin{example}
Let $G=\C\ltimes_{\phi} \C^{2}$ such that \[\phi(z)=\left(
\begin{array}{cc}
e^{z}& 0  \\
0&    e^{-z}  
\end{array}
\right).\]
Then we have $a+\sqrt{-1}b, c+\sqrt{-1}d\in \C$ such that $ \Z(a+\sqrt{-1}b)+\Z(c+\sqrt{-1}d)$ is a lattice in $\C$ and
$\phi(a+\sqrt{-1}b)$ and $\phi(c+\sqrt{-1}d)$
 are conjugate to elements of $SL(4,\Z)$ where we regard  $SL(2,\C)\subset SL(4,\R)$ (see \cite{Hd}).
Hence we have a lattice $\Gamma=(\Z(a+\sqrt{-1}b)+\Z( c+\sqrt{-1}d))\ltimes_{\phi} \Gamma^{\prime\prime}$ such that $\Gamma^{\prime\prime}$ is a lattice of $\C^{2}$.

We take $C=\C$.
For a coordinate $(z_{1},z_{2},z_{3})\in \C\ltimes_{\phi} \C^{2}$,  we have the basis $\frac{\partial}{\partial z_{1}}, e^{z_{1}}\frac{\partial}{\partial z_{2}}, e^{-z_{1}}\frac{\partial}{\partial z_{3}}$ of $\g_{1,0}$ such that $({\rm Ad}_{(z_{1},})_{s}={\rm diag}(1,e^{z_{1}},e^{-z_{1}})$.

 If $b\not \in \pi\Z$ or $c\not \in\pi\Z$, then the condition (${\mathcal E}_{2}$) holds and we have
\[B_{\Gamma}^{1}=\langle d\bar z_{1}\rangle,\]
\[B_{\Gamma}^{2}=\langle d\bar z_{2}\wedge d\bar z_{3}\rangle,\]
\[B_{\Gamma}^{3}=\langle d\bar z_{1} \wedge d\bar z_{2}\wedge d\bar z_{3}\rangle.\]
For the Hermitian metric $g=dz_{1}d\bar z_{1}+e^{-z_{1}-\bar z_{1}}dz_{2}d\bar z_{2}+e^{z_{1}+\bar z_{1}}dz_{3}1d\bar z_{3}$, we have  ${\mathcal H}^{\ast}_{g}(G/\Gamma,J)=C^{\ast}_{\Gamma}$ and the space ${\mathcal H}^{\ast}_{g}(G/\Gamma,J)$ is closed under the Schouten bracket   (see the proof of Corollary \ref{2dee}).
Hence by Proposition \ref{Forsm}, for parameters 
\[t=(t_{12},t_{13},t_{23}, t_{1}^{\bar1},t_{2}^{\bar1}, t_{3}^{\bar1},t^{\bar2\bar3}),\]
considering
\begin{multline*}
\phi(t)=t_{12}e^{z_{1}}\frac{\partial}{\partial z_{1}}\wedge \frac{\partial}{\partial z_{2}}+t_{13}e^{-z_{1}}\frac{\partial}{\partial z_{1}}\wedge \frac{\partial}{\partial z_{3}}+t_{23}\frac{\partial}{\partial z_{1}}\wedge \frac{\partial}{\partial z_{3}}\\
+t_{1}^{\bar1}\frac{\partial}{\partial z_{1}}\wedge d\bar z_{1}+t_{2}^{\bar1}e^{z_{1}}\frac{\partial}{\partial z_{2}}\wedge d\bar z_{1}+t_{3}^{\bar1}e^{-z_{1}}\frac{\partial}{\partial z_{3}}\wedge d\bar z_{1}\\
+t^{\bar2\bar3}d\bar z_{2}\wedge d\bar z_{3},
\end{multline*}
we have 
\[
\Kur^{gen}(G/\Gamma, J)
=\{t=(t_{12},t_{13},t_{23}, t_{1}^{\bar1},t_{2}^{\bar1}, t_{3}^{\bar1},t^{\bar2\bar3})\vert [\phi(t)\bullet\phi(t)]=0\}.
\]
Thus $\Kur^{gen}(G/\Gamma, J)$ is defined by the following quadratic polynomial equations
\[t_{12}t_{13}=0,\, t_{12}t_{1}^{\bar1}=0,\,  t_{12}t_{3}^{\bar1}+t_{13}t_{2}^{\bar1}=0,\, t_{13}t_{1}^{\bar1}=0.
\]
\begin{remark}
In this case, the ordinary Kuranishi space $\Kur(G/\Gamma, J)$ is given by
\[
\Kur(G/\Gamma, J)
=\{t=(0,0,0, t_{1}^{\bar1},t_{2}^{\bar1}, t_{3}^{\bar1},0)\vert [\phi(t)\bullet\phi(t)]=0\}
\]
and hence $\Kur(G/\Gamma, J)$ is unobstructed.
But $\Kur^{gen}(G/\Gamma, J)$ is obstructed.
\end{remark}

If $b \in \pi\Z$ and $c \in\pi\Z$, then  (${\mathcal E}_{2}$) does not hold.
In this case we have $B^{\ast}_{\Gamma}=\textstyle\bigwedge\langle d\bar z_{1}, e^{-z_{1}}d\bar z_{2}, e^{z_{1}}d\bar z_{3}\rangle$.
Hence we have 
\[C^{p,q}_{\Gamma}= \textstyle\bigwedge^{p} \g_{1,0} \otimes  \bigwedge^{q}\langle d\bar z_{1}, e^{-z_{1}}d\bar z_{2}, e^{z_{1}}d\bar z_{3}\rangle.\]
For the Hermitian metric $g=dz_{1}d\bar z_{1}+e^{-z_{1}-\bar z_{1}}dz_{2}d\bar z_{2}+e^{z_{1}+\bar z_{1}}dz_{3}1d\bar z_{3}$, we have  ${\mathcal H}^{\ast}_{g}(G/\Gamma,J)=C^{\ast}_{\Gamma}$ and the space ${\mathcal H}^{\ast}_{g}(G/\Gamma,J)$ is closed under the Schouten bracket   (see the proof of Corollary \ref{2dee}).
Hence by Proposition \ref{Forsm}, for parameters 
\[t=(t_{12},t_{13},t_{23}, t_{1}^{\bar1},t_{1}^{\bar2},t_{1}^{\bar3},t_{2}^{\bar1}, t_{2}^{\bar2},t_{2}^{\bar3},t_{3}^{\bar1},t_{3}^{\bar 2},t_{3}^{\bar3}, t^{\bar1\bar2},t^{\bar1\bar3},t^{\bar2\bar3}),\]
considering 
\begin{multline*}
\phi(t)=t_{12}e^{z_{1}}\frac{\partial}{\partial z_{1}}\wedge \frac{\partial}{\partial z_{2}}+t_{13}e^{-z_{1}}\frac{\partial}{\partial z_{1}}\wedge \frac{\partial}{\partial z_{3}}+t_{23}\frac{\partial}{\partial z_{1}}\wedge \frac{\partial}{\partial z_{3}}\\
+t_{1}^{\bar1}\frac{\partial}{\partial z_{1}}\wedge d\bar z_{1}+t_{1}^{\bar2}e^{-z_{1}}\frac{\partial}{\partial z_{1}}\wedge d\bar z_{2}+t_{1}^{\bar3}e^{z_{1}}\frac{\partial}{\partial z_{1}}\wedge d\bar z_{3}\\
+t_{2}^{\bar1}e^{z_{1}}\frac{\partial}{\partial z_{2}}\wedge d\bar z_{1}+t_{2}^{\bar2}\frac{\partial}{\partial z_{2}}\wedge d\bar z_{2}+t_{2}^{\bar3}e^{2z_{1}}\frac{\partial}{\partial z_{2}}\wedge d\bar z_{3}\\
+t_{3}^{\bar1}e^{-z_{1}}\frac{\partial}{\partial z_{3}}\wedge d\bar z_{1}+t_{3}^{\bar2}e^{-2z_{1}}\frac{\partial}{\partial z_{3}}\wedge d\bar z_{2}+t_{3}^{\bar3}\frac{\partial}{\partial z_{3}}\wedge d\bar z_{3}\\
+t^{\bar1\bar2}e^{-z_{1}}d\bar z_{1}\wedge d\bar z_{2}+t^{\bar1\bar3}e^{z_{1}}d\bar z_{1}\wedge d\bar z_{3}+t^{\bar2\bar3}d\bar z_{2}\wedge d\bar z_{3},
\end{multline*}
we have 
\begin{multline*}
\Kur^{gen}(G/\Gamma, J)\\
=\{t=(t_{12},t_{13},t_{23}, t_{1}^{\bar1},t_{1}^{\bar2},t_{1}^{\bar3},t_{2}^{\bar1}, t_{2}^{\bar2},t_{2}^{\bar3},t_{3}^{\bar1},t_{3}^{\bar 2},t_{3}^{\bar3}, t^{\bar1\bar2},t^{\bar1\bar3},t^{\bar2\bar3})\vert [\phi(t)\bullet\phi(t)]=0\}.
\end{multline*}
Thus $\Kur^{gen}(G/\Gamma, J)$ is defined by the following quadratic polynomial equations (compare with \cite[(3.3)]{Na}):
\begin{eqnarray*}
t_{12}t_{13}=0,\, t_{12}t_{1}^{\bar1}=0,\, t_{12}t_{1}^{\bar2}=0,\, t_{12}t_{3}^{\bar1}+t_{13}t_{2}^{\bar1}=0,\\
t_{12}t_{3}^{\bar2}=0, \, t_{12}t^{\bar1\bar2}-t_{1}^{\bar2}t_{2}^{\bar1}=0,\, -t_{12}t^{\bar1\bar3}+2t_{1}^{\bar1}t_{2}^{\bar3}-t_{1}^{\bar3}t_{2}^{\bar1}=0,\\
t_{13}t_{1}^{\bar1}=0,\, t_{13}t_{1}^{\bar3} =0,\, t_{13}t_{2}^{\bar3}=0,\, 
\\
t_{13}t^{\bar1\bar2}-2t_{1}^{\bar1}t_{3}^{\bar2}+t_{1}^{\bar2}t_{3}^{\bar 1}=0,\, t_{1}^{\bar3}t_{3}^{\bar1}=0,\,t_{1}^{\bar1}t_{1}^{\bar2}=0,\, t_{1}^{\bar1}t_{1}^{\bar3}=0,\\
t_{1}^{\bar2}t_{1}^{\bar3}=0,\, t_{1}^{\bar2}t_{2}^{\bar3}=0,\, t_{1}^{\bar3}t_{3}^{\bar2}=0,\, t_{1}^{\bar2}t^{\bar1\bar3}+t_{1}^{\bar3}t^{\bar1\bar2}=0.
\end{eqnarray*}

In this case, the condition (${\mathcal E}_{2}$) does not hold.
It is known that there exists a small deformation of $J$ which can not be induced by a $G$-left-invariant complex structure on $G$
(see \cite{Na}, \cite{Hd}).

We consider the holomorphic Poisson bi-vector field
\[\mu=e^{z_{1}}\frac{\partial}{\partial z_{1}}\wedge \frac{\partial}{\partial z_{2}}.
\]
Then by Theorem \ref{popar}, the cohomology $H^{\ast}(G/\Gamma, \mu)$ is isomorphic to the cohomology of the cochain complex $(C_{\Gamma}^{\ast}, [\mu\bullet])$.
We compute:
\[H^{1}(G/\Gamma, \mu)=\left\langle \left[e^{z_{1}}\frac{\partial}{\partial z_{2}}\right],\, \left[d\bar z_{1}\right]\right\rangle ,
\]
\[H^{2}(G/\Gamma, \mu)=\left\langle \left[e^{z_{1}}\frac{\partial}{\partial z_{1}}\wedge d\bar z_{3}\right],\,\left[e^{z_{1}}\frac{\partial}{\partial z_{2}}\wedge d\bar z_{1}\right], \left[\frac{\partial}{\partial z_{3}}\wedge d\bar z_{3}\right],  \left[d\bar z_{2}\wedge d \bar z_{3}\right]\right\rangle, 
\]
\begin{multline*}
H^{3}(G/\Gamma, \mu)=\left\langle
 \left[e^{2z_{1}}\frac{\partial}{\partial z_{1}}\wedge \frac{\partial}{\partial z_{2}}\wedge d\bar z_{3}\right],\,\left[e^{z_{1}}\frac{\partial}{\partial z_{2}}\wedge \frac{\partial}{\partial z_{3}}\wedge d\bar z_{3}\right], \right.
\\
\left. \left[e^{z_{1}}\frac{\partial}{\partial z_{1}}\wedge d\bar z_{1}\wedge d\bar z_{3}\right],
 \left[e^{z_{1}}\frac{\partial}{\partial z_{2}}\wedge d\bar z_{2}\wedge d \bar z_{3}\right], \left[d\bar z_{1}\wedge d\bar z_{2}\wedge d \bar z_{3}\right]
\right\rangle, 
\end{multline*}
\[H^{4}(G/\Gamma, \mu)=\left\langle
 \left[e^{2z_{1}}\frac{\partial}{\partial z_{1}}\wedge \frac{\partial}{\partial z_{2}}\wedge d \bar z_{1}\wedge d\bar z_{3}\right],\,\left[e^{z_{1}}\frac{\partial}{\partial z_{2}}\wedge d \bar z_{1}\wedge d \bar z_{2} \wedge d\bar z_{3}\right],\right\rangle,
\]
\[H^{5}(G/\Gamma, \mu)=0,
\]
\[H^{6}(G/\Gamma, \mu)=0.
\]
Obviously the cohomology $H^{\ast}(G/\Gamma, \mu)$ is different from  both the de Rham cohomology and the Dolbeault cohomology. 

\begin{remark}
It is shown that the holomorphic Poisson cohomology of a complex parallelizable nilmanifold $G/\Gamma$  with a holomorphic Poisson structure $\mu$ is decomposed as
\[H^{r}(G/\Gamma, \mu)\cong \bigoplus_{p+q=r}  H^{p}_{[\mu\bullet]}(\g_{1,0})\otimes H^{q}_{\bar\partial} (\g_{0,1}^{\ast})
\]
(see \cite{CGP}).
On the above computations, it is difficult to find such decomposability.
By our example, we can say that the holomorphic Poisson cohomology of a complex parallelizable solvmanifold is more complicated than the holomorphic Poisson cohomology of a complex parallelizable nilmanifold.
\end{remark}
\end{example}

\section{Metabelian case}\label{FORMA}
We consider the following assumption.
\begin{Assumption}\label{asasa}
 $\bullet$ $G=\C^{n}\ltimes_{\phi} \C^{m}$ with a semi-simple action $\phi\colon \C^{n}\to GL_{m}(\C)$ (not necessarily holomorphic).
Then we have a coordinate $z_{1},\dots ,z_{n}, w_{1},\dots ,w_{m}$ of $\C^{n}\ltimes_{\phi} \C^{m}$ such that 
\[\phi(z_{1},\dots ,z_{n})(w_{1},\dots ,w_{m})=(\alpha_{1} w_{1},\dots ,\alpha_{m} w_{m})
\]
where $\alpha_{i}$ are $C^{\infty}$-characters of $\C^{n}$.\\
$\bullet$ $G$ has a lattice $\Gamma$ and
 a lattice $\Gamma$ of  $G=\C^{n}\ltimes_{\phi} \C^{m}$ is the form $\Gamma^{\prime}\ltimes_{\phi} \Gamma^{\prime\prime}$ such that $\Gamma^{\prime}$ and $\Gamma^{\prime\prime}$ are lattices of $\C^{n}$ and $\C^{m}$ respectively and the action $\phi$ of $\Gamma^{\prime}$ preserves $\Gamma^{\prime\prime}$.\\
$\bullet$ $\alpha_{[m]}=1$.
\end{Assumption}
Then this assumption is a special case of Assumption \ref{Ass} such that $(N,J)$ is a complex Abelian Lie group  with Assumption \ref{unim}.

We take unitary characters $\beta_{i}$ and $\gamma_{i}$ such that $\alpha_{i}\beta_{i}^{-1}$ and $\bar\alpha\gamma^{-1}_{i}$ are holomorphic.
In this case, $C^{p,q}_{\Gamma} $ as in Theorem \ref{DGAI} is given by
\[C^{p,q}_{\Gamma}=\left\langle \frac{\partial}{\partial z_{I}}\wedge \alpha_{J}\beta^{-1}_{J}\frac{\partial}{\partial w_{J}}\otimes d\bar z_{K}\wedge \bar \alpha^{-1}_{L}\gamma_{L}d\bar w_{L}{\Big \vert} \begin{array}{cc}\vert I\vert+\vert K\vert=p,\, \vert J\vert+\vert L\vert=q \\ (\beta_{J}\gamma_{L})_{\vert_\Gamma}=1\end{array}\right\rangle.
\]

By Theorem \ref{DGAI}, we prove the following corollary.
\begin{corollary}\label{2dee}
Let $(G,\Gamma)$  be  as in Assumption \ref{asasa}.
Consider the Hermitian metric $g=\sum dz_{i}  d\bar z_{i}+  \sum \alpha^{-1}_{i}\bar\alpha^{-1}_{i}dw_{i}d\bar w_{i}$.
Then we have ${\mathcal H}^{\ast}_{g}(G/\Gamma,J)=C^{\ast}_{\Gamma}$ and hence   ${\mathcal H}^{\ast}_{g}(G/\Gamma,J)$ is a sub-DGA of $dG^{\ast}(G/\Gamma,J)$.
In particular, by Proposition \ref{Forsm},  $\Kur^{gen}(G/\Gamma, J)$ is cut out by polynomial equations of degree 
at most $2$.
\end{corollary}
\begin{proof}
Since for each mult-indices $I$, $J$, $K$ and $L$ the characters $\alpha_{J}\beta^{-1}_{J}$ and $\bar \alpha^{-1}_{L}\gamma_{L}$ are holomorphic,
the operator $\bar \partial$ on  $C^{\ast}_{\Gamma}$ is $0$.

For the $\C$-anti-linear Hodge star operator 
\[\bar\ast_{g}\colon  A^{0,q}(G/\Gamma,\textstyle\bigwedge^{p} T^{1,0}G/\Gamma)\to A^{0,n+m-q}(G/\Gamma,\bigwedge^{n+m-p} T^{1,0}G/\Gamma),\]
 we have $\bar\ast_{g}(C_{\Gamma}^{\ast})\subset C_{\Gamma}^{\ast}$ as Section \ref{spdeff}.
Hence we have $\bar \ast_{g}\bar\partial\bar\ast_{g}(C^{\ast}_{\Gamma})=0$.
Hence we have $C^{\ast}_{\Gamma}\subset {\mathcal H}^{\ast}_{g}(G/\Gamma,J)$.
By Theorem \ref{DGAI},
we have ${\mathcal H}^{\ast}_{g}(G/\Gamma,J)=C^{\ast}_{\Gamma}$.
Hence the corollary follows
\end{proof}

We have:
\begin{theorem}\label{SMSS}
Let $(G,\Gamma)$  be  as in Assumption  \ref{asasa}.
We assume that for any $J,L\subset [m]$,  $(\beta^{-1}_{J}\gamma_{L})_{\vert_{\Gamma}}=1$ if and only if $(\alpha_{J}\bar\alpha^{-1}_{L})_{\vert_{\Gamma}}=1$.
Then we have

$\bullet$ ${\mathcal H}^{\ast}_{g}(G/\Gamma,J)$ is a sub-DGA of $dG^{\ast}(G/\Gamma,J)$ with trivial  brackets.
In particular, by Proposition \ref{Forsm},  $\Kur^{gen}(G/\Gamma,J)$ is smooth.

$\bullet$ Let $\mu \in C^{2,0}_{\Gamma}$ be a holomorphic Poisson bi-vector field.
Then the holomorphic Poisson cohomology $H^{\ast}(G/\Gamma,\mu)$ is isomorphic to the  cohomology 
$H^{\ast}(dG^{\ast}(G/\Gamma,J))$.

\end{theorem}
\begin{proof}
We have
\[C^{p,q}_{\Gamma}=\left\langle \frac{\partial}{\partial z_{I}}\wedge \frac{\partial}{\partial w_{J}}\wedge d\bar z_{K}\wedge d\bar w_{L}{\Big \vert} \begin{array}{cc}\vert I\vert+\vert K\vert=p,\, \vert J\vert+\vert L\vert=q  \\ ( \alpha_{J}\bar \alpha^{-1}_{L})_{\vert_{\Gamma}}=1\end{array}\right\rangle.
\]
Hence the first assertion follows.

For a  holomorphic Poisson bi-vector field $\mu\in C^{2,0}_{\Gamma}$, since $C^{p,q}$ is written as above, we can write
\[\mu= \sum_{\vert I\vert+\vert J\vert =2} a_{IJ} \frac{\partial}{\partial z_{I}}\wedge \frac{\partial}{\partial w_{J}}\]
for $a_{IJ}\in \C$.
Then the operator $[\mu\bullet]\colon C^{\ast,\ast}_{\Gamma}\to C^{\ast+1,\ast}_{\Gamma}$ is trivial.
Hence the second assertion follows.

\end{proof}

\begin{example}\label{TYY}
Let $H=\R^{n}\ltimes_{\psi} \R^{n+1}$ such that 
\[\psi(s_{1},\dots, s_{n})(t_{1},\dots, t_{n}, t_{n+1})=(e^{s_{1}}t_{1},\dots e^{s_{n}}t_{n}, e^{-s_{1}-\dots-s_{n}}t_{n+1}).
\]
For a totally real  algebraic number field $K$ with degree $n+1$, we can construct a lattice $\Delta$ in $H$ by the following way.
Let ${\mathcal O}_{K}$ be the subring of all algebraic integers in $K$ and ${\mathcal O}_{K}^{\ast}$ the unit group of ${\mathcal O}_{K}$. 
Then it is known that ${\mathcal O}_{K}$ can be regarded as a lattice in $\R^{n+1}$.
 By Dirichlet's unit theorem, we have a subgroup $\Gamma^{\prime}_{1}\subset {\mathcal O}_{K}^{\ast}$ such that 
 $\Gamma^{\prime}_{1}$ can be regarded as a lattice in $\R^{n}$ and the semi-direct product $ \Gamma^{\prime}_{1}\ltimes {\mathcal O}_{K}$ can be regarded as a lattice in  $H=\R^{n}\ltimes_{\psi} \R^{n+1}$ (see \cite[Section 3]{TY} for the detail).

Consider $G=\C^{n}\ltimes_{\phi}\C^{2n+2}$ such that 
\begin{multline*}
\phi(z_{1},\dots, z_{n})(w_{1,1},\dots,w_{1,n},w_{1,n+1},w_{2,1},\dots,w_{2,n},w_{2,n+1})\\
=(e^{x_{1}}w_{1,1},\dots, e^{x_{n}}w_{1,n}, e^{-x_{1}-\dots-x_{n}}w_{1,n+1},\\
 e^{-x_{1}}w_{2,1}, \dots, e^{-x_{n}}w_{2,n}, e^{x_{1}+\dots+x_{n}}w_{2,n+1})
\end{multline*}
for complex coordinate 
\[(z_{1}=x_{1}+\sqrt{-1}y_{1},\dots ,z_{n}=x_{n}+\sqrt{-1}y_{n},w_{1},\dots, w_{n+1}, w_{n+2},\dots, w_{2n+2}  ).\]
 By the above argument,
we obtain a lattice $\Gamma=(\Gamma^{\prime}_{1}+\Gamma^{\prime}_{2})\ltimes \Gamma^{\prime\prime}$   by using
a totally real  algebraic number field $K$ with degree $n+1$.
We have
\begin{multline*}(\beta_{1},\dots, \beta_{n}, \beta_{n+1},\beta_{1}^{-1},\dots,\beta_{n}^{-1},\beta_{n+1}^{-1})\\
=(e^{-\sqrt{-1}y_{1}},\dots, e^{-\sqrt{-1}y_{n}}, e^{\sqrt{-1}y_{1}+\dots+\sqrt{-1}y_{n}},\\
 e^{\sqrt{-1}y_{1}}, \dots, e^{\sqrt{-1}y_{n}}, e^{-\sqrt{-1}y_{1}-\dots-\sqrt{-1}y_{n}}).
\end{multline*}
By this we can take $ \Gamma^{\prime}_{2}$ such that  the assumptions of Theorem \ref{SMSS} hold. 
For example, take $\Gamma^{\prime}_{2}=\sqrt{-1}\Z^{n}$.
Hence for such $\Gamma$, $\Kur^{gen}(G/\Gamma,J)$ is smooth.
\end{example}

\begin{remark}
In \cite{Hn}, Hasegawa  showed that  a simply connected solvable Lie group $G$ with a  lattice $\Gamma$ such that $G/\Gamma$ admits a K\"ahler structure can be written as $G=\R^{2k}\ltimes _{\phi}\C^{l}$ such that
\[\phi(t_{j})((z_{1},\dots ,z_{l}))=(e^{\sqrt{-1}\theta^{j}_{1}t_{j}}z_{1},\dots ,e^{\sqrt{-1}\theta^{j}_{l}t_{j}}z_{l}),
\]
where each $e^{\sqrt{-1}\theta^{j}_{i}}$ is a root of unity.
In particular if $G$ is completely solvable and a solvmanifold $G/\Gamma$ admits a K\"ahler structure, then $G$ is abelian.
Hence  Example \ref{TYY} does not admit a  K\"ahler structure.
We notice that   Example \ref{TYY} satisfies the Hodge symmetry and decomposition (see \cite{KH}).
\end{remark}


\begin{thebibliography}{40}
 \bibitem{BG} C. Benson, and C. S. Gordon,  K\"ahler and symplectic structures on nilmanifolds, Topology {\bf 27} (1988), no. 4, 513--518.
\bibitem{BK}S. Barannikov, M. Kontsevich, Frobenius manifolds and formality of Lie algebras of polyvector fields,Int. Math. Res. Not. IMRN (1998), no. 4 201–215. 
\bibitem{CG}
C. G. Cavalcanti, M. Gualtieri. Generalized complex structures on nilmanifolds, J. Symplectic
Geom. {\bf 2} (2004), no. 3, pp. 393--410.
\bibitem{CFI} S. Console, A. Fino, Dolbeault cohomology of compact nilmanifolds. Transform. Groups {\bf 6} (2001), no. 2, 111--124.
\bibitem{CFP} S. Console, A. Fino,  Y. S.  Poon,
Stability of abelian complex structures. 
Internat. J. Math. {\bf 17} (2006), no. 4, 401--416. 
\bibitem{CF} L. A. Cordero, M. Fern\'andez, A. Gray, L. Ugarte, Compact nilmanifolds with nilpotent complex structures: Dolbeault cohomology. Trans. Amer. Math. Soc. {\bf 352} (2000), no. 12, 5405--5433.
\bibitem{CGP}  Z. Chen, D.  Grandini  Y. S. Poon,  Holomorphic Poisson cohomology. Complex Manifolds {\bf 2} (2015), 34--52. 
\bibitem{dek} K.  Dekimpe,
Semi-simple splittings for solvable Lie groups and polynomial structures. Forum Math. {\bf 12} (2000), no. 1, 77--96.
\bibitem{Gua} M. Gualtieri,
Generalized complex geometry. Ann. of Math. (2) {\bf 174} (2011), no. 1, 75--123. 
98), no. 1, 47--92.
%
\bibitem{H} K. Hasegawa, Minimal models of nilmanifolds. \textit{Proc. Amer. Math. Soc.} {\bf 106} (1989), no. 1, 65--71. 

\bibitem{Hn} K. Hasegawa, A note on compact solvmanifolds with K\"ahler structures. Osaka J. Math. {\bf 43} (2006), no. 1, 131--135.
\bibitem{Hd} K. Hasegawa, Small deformations and non-left-invariant complex structures on six-dimensional compact solvmanifolds. Differential Geom. Appl. {\bf 28} (2010), no. 2, 220--227.

\bibitem{Kd}
H. Kasuya, Techniques of computations of Dolbeault cohomology of  solvmanifolds.    Math. Z.  {\bf 273}, (2013), 437--447.
\bibitem{KH} H. Kasuya, Hodge symmetry and decomposition on non-K\"ahler
solvmanifolds.   J. Geom. Phys. {\bf 76}, 61--65 (2014)
\bibitem{KDD}
H. Kasuya, de Rham and Dolbeault Cohomology of solvmanifolds with local systems.    Math. Res. Lett. {\bf 21} (2014), no. 4, 781--805.
\bibitem{Ku}
M. Kuranishi, On the locally complete families of complex analytic structures. Ann. of Math. (2), {\bf 75} (1962), 536--577.
\bibitem{L-G}
C.  Laurent-Gengoux, M. Stienon, P. Xu, Holomorphic Poisson Manifolds and Holomorphic Lie Algebroids, Int. Math. Res. Not. IMRN (2008), Art. ID rnn 088, 46 pp. 
\bibitem{Mc}
J. McCleary, A user’s guide to spectral sequences, Second edition, Cambridge Studies in Advanced Mathematics, {\bf 58}, Cambridge
University Press, Cambridge, 2001.
\bibitem{MPPS}
C. Maclaughlin, H. Pedersen, Y. S. Poon and S. Salamon, Deformation of 2-step nilmanifolds with abelian complex structures, J. London Math. Soc. (2) {\bf 73} (2006) 173--193. 
\bibitem{Na} I. Nakamura, {\it Complex parallelisable manifolds and their small deformations}, J. Differential Geometry {\bf 10} (1975), 85--112. 
\bibitem{Nom}K. Nomizu, On the cohomology of compact homogeneous spaces of nilpotent Lie groups. Ann. of Math. (2) {\bf 59}, (1954). 531--538.

 \bibitem{R}
 M.S. Raghunathan, Discrete subgroups of Lie Groups, Springer-Verlag, New York, 1972. Ergebnisse der Mathematik und ihrer Grenzgebiete, Band 68.
 \bibitem{RO}
S. Rollenske, Lie-algebra Dolbeault-cohomology and small deformations of nilmanifolds. J. Lond. Math. Soc. (2) {\bf 79} (2009), no. 2, 346--362.
\bibitem{ROc}S. Rollenske, 
Dolbeault cohomology of nilmanifolds with left-invariant complex structure. Complex and differential geometry, 369--392, Springer Proc. Math., {\bf 8}, Springer, Heidelberg, 2011.
\bibitem{ROP}
S. Rollenske,  The Kuranishi space of complex parallelisable nilmanifolds. J. Eur. Math. Soc. (JEMS) {\bf 13} (2011), no. 3, 513--531.
\bibitem{Sak} Y. Sakane, On compact complex parallelisable solvmanifolds. Osaka J. Math. {\bf 13} (1976), no. 1, 187--212.
\bibitem{sawai-yamada}
H. Sawai, T. Yamada,  Lattices on Benson-Gordon type solvable Lie groups. Topology Appl. {\bf 149} (2005), no. 1-3, 85--95.
\bibitem{Sal} S. M. Salamon,
Complex structures on nilpotent Lie algebras. 
J. Pure Appl. Algebra  {\bf157} (2001), no. 2-3, 311--333. 
\bibitem{TY}
N. Tsuchiya, A. Yamakawa, Lattices of some solvable Lie groups and actions of products of affine groups. Tohoku Math. J. (2) {\bf 61} (2009), no. 3, 349--364. 
\bibitem{Va} I. Vaisman, Lectures on the geometry of Poisson manifolds. Progress in Mathematics, 118. Birkhäuser Verlag, Basel, 1994.
\end{thebibliography}
\end{document}